\documentclass{amsart}

\usepackage{mdframed} 
\usepackage[margin=1.5in]{geometry}
\usepackage{latexsym,amsxtra,amscd,ifthen,amsmath,color, amsthm}
\usepackage{amsfonts}
\usepackage{soul}

\makeatletter
\@namedef{subjclassname@2020}{%
  \textup{2020} Mathematics Subject Classification}
\makeatother

\usepackage{verbatim}
\usepackage{hyperref}
\usepackage{cancel}
\usepackage{amsthm, tikz, tikz-cd, bbm} 
\usetikzlibrary{arrows}
\usepackage[capitalise]{cleveref}
\usepackage{amssymb}
\usepackage[notcite,notref,final]{showkeys}
\usepackage[all,cmtip]{xy}
\usepackage[normalem]{ulem} 
\usepackage[shortlabels]{enumitem}
\usepackage{xcolor}
\definecolor{darkgreen}{RGB}{55,138,0}
\definecolor{burntorange}{RGB}{180,85,0}
\definecolor{navyblue}{RGB}{18,40,180}
\definecolor{cyan(process)}{rgb}{0.0, 0.6, 1.0}

\hypersetup{
 colorlinks  = true,  
 urlcolor   = blue,  
 linkcolor  = navyblue,  
 citecolor  = black   
}
\allowdisplaybreaks

\numberwithin{equation}{section}

\theoremstyle{plain}
\newtheorem{theorem}{Theorem}[section]
\newtheorem*{theorem*}{Theorem}
\newtheorem*{definition*}{Definition}
\newtheorem{lemma}[theorem]{Lemma}

\newtheorem{proposition}[theorem]{Proposition}
\newtheorem{corollary}[theorem]{Corollary}

\newtheorem{definitionproposition}[theorem]{Definition-Proposition}

\theoremstyle{definition}
\newtheorem{definition}[theorem]{Definition}
\newtheorem{notation}[theorem]{Notation}

\newtheorem{example}[theorem]{Example}
\newtheorem{hypothesis}[theorem]{Hypothesis}

\newtheorem{remark}[theorem]{Remark}
\newtheorem{question}[theorem]{Question}

\newcommand{\stkout}[1]{\ifmmode\text{\sout{\ensuremath{#1}}}\else\sout{#1}\fi}

\newcommand{\kk}{\Bbbk}
\renewcommand{\Vec}{{\sf Vec}_{\kk}}
\newcommand{\C}{\mathcal{C}}
\newcommand{\D}{\mathcal{D}}
\newcommand{\M}{\mathcal{M}}
\newcommand{\unit}{\mathbbm{1}}
\newcommand{\End}{\textnormal{End}}
\newcommand{\Rep}{{\sf Rep}}
\newcommand{\ncat}{\underrightarrow{\mathbb{N}_0}}
\newcommand{\Ncat}{\underline{\mathbb{N}_0}}

\newcommand{\ess}{\mathcal{S}}
\newcommand{\ep}{\varepsilon}
\newcommand{\ev}{\textnormal{ev}}
\newcommand{\coev}{\textnormal{coev}}
\newcommand{\coker}{\textnormal{coker}}
\newcommand{\leftdual}{{}^* \hspace{-.08cm}}
\newcommand{\one}{\mathbbm{1}}
\newcommand{\gr}{\textnormal{gr}}
\newcommand{\id}{\textnormal{id}}
\newcommand{\RN}[1]{%
  \textup{\uppercase\expandafter{\romannumeral#1}}%
}

\newcommand{\pullbackcorner}[1][dr]{\save*!/#1-1.5pc/#1:(-1,1)@^{|-}\restore}

\newcommand{\pushoutcorner}[1][ul]{\save*!/#1-1.5pc/#1:(-1,1)@^{|-}\restore}

\makeatletter              
\let\c@equation\c@theorem  
\makeatother

\begin{document}

\title[Filtered Frobenius algebras in monoidal categories]
{Filtered Frobenius algebras in monoidal categories}

\author{Chelsea Walton and Harshit Yadav}

\address{Walton: Department of Mathematics, Rice University, Houston, TX 77005, USA}
\email{notlaw@rice.edu}

\address{Yadav: Department of Mathematics, Rice University, Houston, TX 77005, USA}
\email{hy39@rice.edu}


\begin{abstract} 
We develop filtered-graded techniques for algebras in monoidal categories with the main goal of establishing a categorical version of Bongale's 1967 result: A filtered deformation of a Frobenius algebra over a field is Frobenius as well. Towards the goal, we first  construct a monoidal associated graded functor, building on prior works of Ardizzoni-Menini,  of Galatius et al., and of Gwillian-Pavlov. Next, we produce equivalent conditions for an algebra in a rigid monoidal category to be Frobenius in terms of the existence of categorical Frobenius form; this builds on work of Fuchs-Stigner. These two results of independent interest are then used to achieve our goal. 
As an application of our main result, we show that any exact module category over a symmetric finite tensor category $\mathcal{C}$ is represented by a Frobenius algebra in $\mathcal{C}$. Several directions for further investigation are also proposed.
\end{abstract}


\maketitle

\setcounter{tocdepth}{2}
\tableofcontents


\section{Introduction} \label{sec:intro}

 This article is a study of filtered-graded techniques for  algebras in monoidal categories $(\C, \otimes, \unit)$.  Unless stated otherwise, we assume that all algebras $A$ are $\mathbb{N}_0$-filtered, where $\mathbb{N}_0$ is the monoid of natural numbers including 0, with filtration $F_A$ so that $A = \bigcup_{i \in \mathbb{N}_0} F_A(i)$.  We say that a filtered algebra $A$ is {\it connected} if $F_A(0) = \unit$, and we refer to  $A$ as a {\it filtered deformation} of its associated graded algebra gr$(A)$. This terminology is standard for the monoidal category, $\Vec$, of finite-dimensional  vector spaces over a field $\kk$, and the framework for filtered and graded algebras in general monoidal categories is developed further in this work.
To motivate our main result,  recall that in $\Vec$ many algebraic properties of (graded) algebras lift to filtered deformations, including being an integral domain, prime,   Noetherian \cite[Section~1.6]{MR}, and in some cases, being Calabi-Yau \cite{BT,WZ}. Here, we investigate when the Frobenius condition for graded algebras in certain monoidal categories lifts  to  filtered deformations. Namely, our goal  is to generalize the following result of P. R. Bongale for algebras in $\Vec$.

\begin{theorem} \cite[Theorem~2]{Bongale1}
Let $A$ be a finite-dimensional, connected, filtered $\kk$-algebra. If the associated graded algebra $\gr(A)$ is a Frobenius $\kk$-algebra, then so is~$A$.
\end{theorem}

We achieve a generalization of her result for algebras in abelian, rigid  monoidal categories, which includes algebras in $\Vec$ (i.e., $\kk$-algebras), and algebras in categories of finite-dimensional representations of finite-dimensional (weak, quasi-)Hopf algebras $H$ over $\kk$ (i.e., $H$-module algebras over $\kk$). 

\begin{theorem}[Theorem~\ref{thm:main}] \label{thm:main-intro}
Let $\C$ be an abelian, rigid monoidal category, and let $A$ be a connected filtered algebra in $\C$ with finite monic filtration.  If the associated graded algebra $\gr(A)$ is a Frobenius algebra in $\C$, then so is $A$.
\end{theorem}

One application of this theorem is to further the study of open-closed 2-dimensional topological quantum field theories; see \cite{Laz} and \cite[Section~2.4]{LaudaPfeiffer}. Moreover, deformations of Frobenius algebras (over a field) are used to find polynomial solutions to the Witten-Dijkgraaf-Verlinde-Verlinde equation, which in turn  describe the moduli space of topological conformal field theories \cite{Dubrovin}. 

\smallskip

To prove Theorem~\ref{thm:main-intro}, we first present a framework to study monoidal categories ${\sf Gr}(\C)$ (resp., ${\sf Fil}(\C)$) consisting of graded (resp., filtered) objects in $\C$ [Section~\ref{sec:filt-gr}], as well as algebraic structures within these categories [Section~\ref{sec:prelim}, Definition~\ref{def:fil-gr-C}]. Previous works that prompted this framework include \cite{Sch1996}, \cite{BD} \cite{ArMe}, \cite[Section~3.3]{HaMi}, \cite{GwPa}, and \cite[Section~5]{GKR}. Then, the associated graded construction is established in Section~\ref{sec:assgr}, which includes the definition of an associated graded functor and the result below.

\begin{theorem}[Theorem~\ref{thm:gr}, Proposition~\ref{prop:gr-leftadj}] \label{thm:gr-intro} If $\C$ is an abelian, monoidal category with $\otimes$ biexact, then the associated graded functor $\gr: {\sf Fil}(\C) \to {\sf Gr}(\C)$ given in Definition~\ref{def:gr-functor} is  monoidal and is right exact.
\end{theorem}

Thus, the associated graded functor, gr,  yields a canonical graded algebra in $\C$ from a filtered algebra in $\C$. The results in Sections~\ref{sec:monoidal} and~\ref{sec:assgr} also hold for braided monoidal categories $\C$ and (graded, filtered) commutative algebras in $\C$  [Definitions~\ref{defn:brMon}--\ref{defn:brComAlg}].

\smallskip

As an application of the filtered-graded techniques developed in Sections~\ref{sec:monoidal} and~\ref{sec:assgr}, we examine how one could study filtered deformations of graded quotient algebras in monoidal categories in Section~\ref{sec:quotient}. Consider the following result.

\begin{corollary}[Corollary~\ref{cor:quot-gr}] \label{cor:quot-gr-intro}
If $A$ is a filtered algebra in $\C$, and $I$ is a filtered weak ideal of $A$ in~$\C$ \textnormal{[Definition~\ref{def:weakideal}]}, then 
$\gr(A)/\gr(I) \; \cong \; \gr(A/I)$
as graded algebras in $\C$.
\end{corollary}

As in the case for $\C = \Vec$, computing $\gr(I)$ can be tedious; Poincar\'{e}-Birkhoff-Witt theorems and related homological methods are used to address this problem \cite{SW}. It would be interesting to develop such techniques to study filtered deformations of graded quotient algebras in monoidal categories [Remark~\ref{rem:quot-gr}].

\smallskip
Next, in Section~\ref{sec:newDefn}, we establish equivalent conditions for an algebra in an abelian, rigid monoidal category to be Frobenius, building on  previously known equivalent conditions. This is of independent interest due the prevalence of Frobenius algebras in rigid monoidal categories in generalizations of Morita equivalence  \cite{Muger,Yam,MMPRTW}, in computer science \cite{CPV}, and  in topological quantum field theory and conformal field theory \cite{Segal,Moore, KS, Hen}. For the latter, see also \cite{Sch-ICM} for an overview of works by Fuchs-Runkel-Schweigert and others on this topic including \cite{FRS1, FRS2, FRS3, FRS4, FFRS}. In fact, our result below builds on previous work of Fuchs-Stigner \cite{FS}.

\begin{theorem}[Theorem~\ref{thm:defFrob}] \label{thm:defFrob-intro}
Take $\C$ an abelian, rigid monoidal category, and let $A$ be an algebra in~$\C$. We have that $A$ is Frobenius in the sense that it is admits a compatible coalgebra structure as in Definition~\ref{def:algC} if and only if $A$ admits a Frobenius form as in Definition~\ref{def:Frobform}.
\end{theorem}

After we present some preliminary results on Frobenius graded algebras in abelian, rigid monoidal categories in Section~\ref{sec:connected}, Theorems~\ref{thm:gr-intro} and~\ref{thm:defFrob-intro} are then used to achieve  Theorem~\ref{thm:main-intro} in Section~\ref{sec:thmmain}. 
In Section~\ref{sec:questions}, we highlight several directions for further investigation on Theorem~\ref{thm:main-intro}, including connections to \cite{Bongale2, LT}, questions on additional features of the associated graded functor of Theorem~\ref{thm:gr-intro}, and connections to various 2-dimensional topological quantum field theories.

\smallskip

Finally, as an application of our main result, Theorem~\ref{thm:main-intro}, we obtain a result about representations of an important class of abelian, rigid monoidal categories: {\it symmetric finite tensor categories} [Definition~\ref{def:sym-tensor}]. Such categories are  known to be equivalent to the category of super-representations of a {\it finite supergroup} by work of Deligne \cite[Corollary~0.7]{deligne2002categories}. It is useful to study representations of such categories, and of monoidal categories in general, by way of {\it module categories} [Definition~\ref{def:mod-category}]. Key results of Ostrik and Etingof state that well-behaved module categories $\mathcal{M}$ over a large class of monoidal categories $\mathcal{C}$ are equivalent to the category of modules over an algebra $A$ in $\mathcal{C}$ as $\C$-module categories \cite[Theorem~3.1]{ostrik2003module} \cite[Theorem~3.17]{etingof2003finite}. In this case, we say that $A$ {\it represents} $\mathcal{M}$ [Definition-Proposition~\ref{defprop:repAlg}]. When $\mathcal{C}$ is a symmetric finite tensor category, Etingof-Ostrik describes a choice of algebra representatives of the modules categories over $\mathcal{C}$ in terms of {\it internal End objects} \cite[Section~4.2]{etingof2003finite}. We build on this result, and establish the following theorem.

\begin{theorem}[Theorem~\ref{thm:repByFrob}] \label{thm:repByFrob-intro} Every exact module category over a symmetric finite tensor category $\mathcal{C}$ is represented by a Frobenius algebra in $\mathcal{C}$.
\end{theorem}

This result is achieved as the algebra representatives produced in Etingof-Ostrik's work are the combination of an endomorphism algebra and a braided Clifford algebra in $\mathcal{C}$, via the tensor product operation and induction functors. We show that the endomorphism algebra of interest here is a Frobenius algebra, and that tensor product and induction preserve Frobenius algebras. Finally, the proof of Theorem~\ref{thm:repByFrob-intro} is completed by showing that the braided Clifford algebra of interest here is Frobenius by Theorem~\ref{thm:main-intro}: its associated graded algebra is an exterior algebra, which is known to be Frobenius.

\smallskip

Given Theorem~\ref{thm:repByFrob-intro}, it is also natural to ask:

\begin{question}
What are the precise conditions on a monoidal category $\C$ and a $\C$-module category $\M$ for $\M$ to be represented by a Frobenius algebra $A$ in $\C$?
\end{question}

\noindent
In fact, module categories represented by Frobenius algebras play a vital role in 2-dimensional conformal field theory \cite[Section~3]{Sch-ICM}. We leave this question for the reader to consider.

\medskip

\noindent {\bf Acknowledgements.} The authors thank J\"{u}rgen Fuchs for useful comments and questions, and also thank the anonymous referee for their insightful feedback. The second author   was partially supported by the US National Science Foundation grants DMS-1903192, 2100756.


\section{Algebraic structures in monoidal categories} \label{sec:monoidal}
In this section, we discuss monoidal categories, and various algebraic structures in these categories [Section~\ref{sec:prelim}]. We also discuss such structures in the filtered and graded settings  [Section~\ref{sec:filt-gr}]. 

\begin{hypothesis} \label{hyp:abelian}
We assume that all categories in this work are abelian, and all functors between them are additive.
\end{hypothesis}


\subsection{Preliminaries} \label{sec:prelim}

With the exception of Definition~\ref{def:weakideal}, we refer the reader to  \cite[Sections~2.1, 7.8 and~8.1]{EGNO} and \cite{FS} for more details about the structures below. We begin by recalling the categories in which we will work throughout.

\begin{definition}[$\C, \otimes, \unit$]
A \textit{monoidal category} $\mathcal{C}:=(\C, \otimes, \unit, \alpha, l, r)$ consists of the following data: a category $\C$, a bifunctor $\otimes : \C \times \C\rightarrow \C$,
an object $\unit \in \C$, 
a natural associativity isomorphism   $\alpha_{X,X',X''}: (X \otimes X')\otimes X'' \overset{\sim}{\to} X \otimes (X'\otimes X'')$ for each $X,X',X'' \in \C$, 
natural unitality isomorphisms  $l_X: \mathbbm{1} \otimes X \overset{\sim}{\to} X, \:\: r_X:X \otimes \mathbbm{1} \overset{\sim}{\to} X$ for each $X \in \C$,
such that the pentagon and triangle coherence conditions are satisfied.
\end{definition}

\begin{hypothesis} \label{hyp:strict} We assume that all monoidal categories here are/ have: 
\begin{enumerate}
    \item {\it strict}, in the sense that the associativity and unitality isomorphisms are equalities;
    
    \smallskip
    
    \item $\otimes$ is {\it biexact}, in the sense that the functors $(X \otimes -)$ and $(- \otimes X) : \C \to \C$ are exact, for each object $X$ in $\C$.
\end{enumerate}
\end{hypothesis}

\begin{definition}[$F,F_2,F_0$]  \label{def:monfunc}
Let $(\C, \otimes_{\C}, \unit_{\C})$ and  $(\mathcal{D},\otimes_{\mathcal{D}},\unit_{\mathcal{D}})$ be monoidal categories.
A \textit{monoidal functor} $(F, F_2, F_0): \C \to \mathcal{D}$ consists of the following data:
    \begin{itemize}
        \item a functor $F: \C \to \mathcal{D}$,
        \item a natural transformation $F_2= \{F_2(X,X') : F(X)\otimes_{\mathcal{D}} F(X') \to F(X \otimes_{\C} X')\}_{X,X' \in \C}$,
        \item a morphism $F_0: \unit_{\mathcal{D}} \to F(\unit_{\C})$ in $\mathcal{D}$,
    \end{itemize}
    that satisfy the following associativity and unitality constraints, for $X,X',X'' \in \C$:
    \[
    \begin{array}{rl}
    \medskip
    F_2(X,X'\otimes_\C X'')(\id_{F(X)} \otimes_{\mathcal{D}} F_2(X',X'')) &= F_2(X \otimes_\C X',X'')(F_2(X,X')\otimes_{\mathcal{D}}\id_{F(X'')}),\\
    \medskip
    F_2(\unit_\C,X)  (F_0 \otimes_{\mathcal{D}} \id_{F(X)}) &= \id_{F(X)},\\
        \medskip
   F_2(X,\unit_\C)  (\id_{F(X)} \otimes_{\mathcal{D}} F_0)  &= \id_{F(X)}.
    \end{array}
    \]
\end{definition}

Next, we consider certain algebraic structures within monoidal categories.

\begin{definition}[$m, u, {\sf Alg}(\C), \Delta, \ep$] \label{def:algC}
Take $\C$ to be a monoidal category, and consider the (categories of) algebraic structures below.
\begin{enumerate}
\item An \textit{algebra} in $\C$ is a triple $(A,m,u)$ consisting of an object  $A\in \C$, and morphisms $m: A \otimes  A \rightarrow A$, $u: \mathbbm{1} \rightarrow A$  in $\C$, satisfying associativity and unitality constraints: $m(m\otimes \text{id}_A) = m(\text{id}_A \otimes m)$, and $m(u\otimes \text{id}_A) =\text{id}_A$ = $m(\text{id}_A\otimes u)$. A \textit{morphism} of algebras $(A,m_A,u_A)$ and $(B,m_B,u_B)$ is a morphism $f:A\rightarrow B$ in $\C$ so that $fm_A = m_B(f\otimes f)$ and $fu_A=u_B$. 
Algebras and their morphisms in $\C$ form a category, which we denote by $\textsf{Alg}(\C)$.

\medskip

\item A \textit{coalgebra} in $\C$ is a triple $(A,\Delta,\varepsilon)$ consisting of an object  $A \in \C$, and morphisms $\Delta: A \rightarrow A \otimes  A$,  $\varepsilon:  A \to \mathbbm{1}$  in $\C$, satisfying coassociativity and counitality constraints: $(\Delta\otimes \text{id}_A)\Delta = (\text{id}_A \otimes \Delta) \Delta$ and $(\varepsilon\otimes \text{id}_A)\Delta =\text{id}_A$ = $(\text{id}_A\otimes \varepsilon)\Delta$. A \textit{morphism} of coalgebras $(A,\Delta_A,\varepsilon_A)$ and $(B,\Delta_B,\varepsilon_B)$ is a morphism $f:A\rightarrow B$ in $\C$ so that $\Delta_B f = (f\otimes f)\Delta_A$ and $\varepsilon_B f= \varepsilon_A$. Coalgebras  and their morphisms in $\C$ form a category, which we denote by $\textsf{Coalg}(\C)$.

\medskip

\item A \textit{Frobenius algebra} in $\C$ is a 5-tuple $(A,m,u,\Delta,\varepsilon)$ where $(A,m,u) \in {\sf Alg}(\C)$ and $(A,\Delta,\varepsilon) \in {\sf Coalg}(\C)$ so that $(m \otimes \text{id}_A)(\text{id}_A \otimes \Delta) = \Delta m = (\text{id}_A \otimes m)(\Delta \otimes \text{id}_A)$. A \textit{morphism} of Frobenius algebras $f:A\rightarrow B$ is a map  in ${\sf Alg}(\C) \cap {\sf Coalg}(\C)$.
Frobenius algebras and their morphisms in $\C$ form a category.
\end{enumerate}
\end{definition}

\begin{definition}[$\lambda, {}_A\C, \rho, \C_A, {}_A \C_A$] Take a monoidal category $\C$ with an algebra $A:=(A,m,u)$ in $\C$. 
\begin{enumerate}
    \item A \textit{left $A$-module} in $\C$ is a pair $(M,\lambda^A_M:=\lambda_M)$ consisting of  an object $M$ and a morphism $\lambda_M:A\otimes M \rightarrow M$ in $\C$ satisfying $\lambda_M(m\otimes \text{id}_M) = \lambda_M(\text{id}_A \otimes \lambda_M)$ and $\lambda_M(u \otimes \text{id}_M)=\text{id}_M$. A \textit{morphism} of left $A$-modules $(M,\lambda_M) \rightarrow (N,\lambda_N)$ is a morphism $f:M \rightarrow N$ in $\C$ such that $\lambda_N(\text{id}_A \otimes f) = f \lambda_M$. Left $A$-modules and their morphisms form a category, which we will denote by ${}_A\C$.
    
    \medskip
    
    \item  A \textit{right $A$-module} in $\C$ is a pair $(M, \rho_M^A:=\rho_M)$ consisting of  an object $M$ and a morphism $\rho_M^A:=\rho_M:M \otimes A \rightarrow M$ in $\C$ satisfying $\rho_M(\text{id}_M \otimes m) = \rho_M(\rho_M \otimes \text{id}_A)$ and $\rho_M(\text{id}_M \otimes u)=\text{id}_M$. A \textit{morphism} of right $A$-modules $(M,\rho_M) \rightarrow (N,\rho_N)$ is a morphism $f:M \rightarrow N$ in $\C$ such that $\rho_N(f \otimes \text{id}_A) = f \rho_M$. Right $A$-modules and their morphisms form a category, which we will denote by $\C_A$.
    
      \medskip
    
    \item An \textit{$A$-bimodule} in $\C$ is a triple $(M, \lambda_M, \rho_M)$ where $(M,\lambda_M) \in {}_A \C$ and $(M, \rho_M) \in \C_A$, so that $\rho_M(\lambda_M \otimes \id_A) = \lambda_M(\id_A \otimes \rho_M)$. A \textit{morphism} of $A$-bimodules is a morphism in $\C$ that belongs to both ${}_A\C$ and $\C_A$. The collection of $A$-bimodules and their morphisms form a category, which we will denote by ${}_A\C_A$.
\end{enumerate}
\end{definition}

Now we discuss various notions of an ideal of an algebra, including our introduction of (one-sided) weak ideals.

\begin{definition}[$\phi$] \label{def:weakideal}
Take an algebra $A:=(A,m,u)$ in a monoidal category $\C$, and recall that $A$ is in ${}_A \C$ (resp., $\C_A$) by using $\lambda_A^A = m_A$ (resp., using $\rho_A^A = m_A$). 
\begin{enumerate}
    \item An object $(I,\lambda_I) \in {}_A \C$ is said to be {\it left weak ideal} of $A$ if there exists a morphism $\phi_I^A:=\phi_I: I \to A$ in ${}_A \C$. That is, $\phi_I \in \C$ with $ \phi_I \lambda_I  =m_A(\id_A \otimes \phi_I)$.
    
    \medskip
    
    \item An object $(I,\rho_I) \in \C_A$ is said to be {\it right weak ideal} of $A$ if there exists a morphism $\phi_I^A:=\phi_I: I \to A$ in $\C_A$. That is, $\phi_I \in \C$ with $ \phi_I \rho_I  =m_A(\phi_I \otimes \id_A)$.
    
    \medskip
    
    \item A left (right) weak ideal $(I, \lambda_I, \phi_I)$ (resp., $(I, \rho_I, \phi_I)$) is called a {\it left (resp., right) ideal} of $A$ if $\phi_I$ is monic.
    
    \medskip
    
    \item We call an $A$-bimodule $(I, \lambda_I, \rho_I)$ a {\it (weak) ideal} of $A$ if it comes equipped with a morphism $\phi_I: I \to A$ so that $(I, \lambda_I, \phi_I)$ is a left (resp., weak) ideal of $A$ and $(I, \rho_I, \phi_I)$ is a right (resp., weak) ideal of $A$.
\end{enumerate}
\end{definition}

Monoidal functors preserve algebras, modules, and left/right weak ideals as we see below.

\begin{proposition} \label{prop:mon-preserve}
Let $(F,F_2,F_0):\C \to \D$ be a monoidal functor. Take $(A,m,u) \in {\sf Alg}(\C)$.
\begin{enumerate}[(a),font=\upshape]
    \item Then, $(F(A),  F(m)F_2(A,A), F(u)F_0) \in {\sf Alg}(\D)$.
    
    \medskip
    
    \item If $(M,\lambda_M^A) \in {}_A \C$, then $(F(M),\lambda_{F(M)}^{F(A)}) \in {}_{F(A)} \D$, where $\lambda_{F(M)}^{F(A)} = F(\lambda_M^A) F_2(A,M)$. Likewise,  one gets a right $F(A)$-module in $\D$ from a right $A$-module in $\C$.
     
    \medskip
    
    \item If $(I,\lambda_I^A, \phi_I^A)$ is a left weak ideal of $A$, then $(F(I),\lambda_{F(I)}^{F(A)}, \phi_{F(I)}^{F(A)})$ is a left weak ideal of $F(A)$, where $\phi_{F(I)}^{F(A)} = F(\phi_I^A)$. Likewise, one gets a right weak ideal of  $F(A)$ in $\D$ from a right weak ideal of $A$ in $\C$.
    
    \medskip
    
    \item Suppose that $(I,\lambda_I^A, \phi_I^A)$  is a left ideal of $A$. Then $(F(I),\lambda_{F(I)}^{F(A)}, \phi_{F(I)}^{F(A)})$ is a left ideal of $F(A)$ if and only if $F$ preserves monomorphisms. We have a similar statement for the preservation of right ideals.
\end{enumerate}
\end{proposition}

\begin{proof}
Part (a) is well-known; see, e.g., \cite[Corollary~5]{DayPas}. Part (b) follows in a similar manner to part (a). For part (c), it suffices to verify that $\phi_{F(I)}^{F(A)} \in {}_{F(A)} \D$, which we achieve via the following computation:
\[
\begin{array}{rl}
\smallskip
\lambda_{F(A)}^{F(A)}\; (\id_{F(A)} \otimes F(\phi_I^A)) &= m_{F(A)} \; (\id_{F(A)} \otimes F(\phi_I^A)) \\
\smallskip
&= F(m_A) \; F_2(A,A) \; (\id_{F(A)} \otimes F(\phi_I^A))  \\
\smallskip
&= F(\lambda_A^A) \; F_2(A,A) \; (\id_{F(A)} \otimes F(\phi_I^A)) \\
\smallskip
&= F(\lambda_A^A) \; F(\id_A \otimes \phi_I^A) \; F_2(A,I) \\
\smallskip
&=  F(\phi_I^A) \; F(\lambda_I^A) \; F_2(A,I) 
\\
\smallskip
&=  F(\phi_I^A) \; \lambda_{F(I)}^{F(A)}. 
\end{array}
\]
In particular, the fourth equality holds by the naturality of $F_2$ and the fifth equality holds as $\phi_I^A \in {}_A \C$. Part (d) follows from part (c) and the definition of a left (resp., right) ideal.
\end{proof}

Now we turn our attention to braided categories and algebraic structures within them.

\begin{definition}[$c_{X,Y}$] \label{defn:brMon}
A \textit{braiding} on a monoidal category $(\C,\otimes, \one)$ is a natural family of isomorphisms $$c = \{ c_{X,Y}: X\otimes Y \overset{\sim}{\rightarrow} Y \otimes X \}_{X,Y\in \textnormal{Ob}(\C)},$$ such that
$c_{X,Y\otimes Z} = (\id_Y \otimes c_{X,Z})(c_{X,Y} \otimes \id_Z)$ and $c_{X\otimes Y,Z} = (c_{X,Z} \otimes \id_Y)(\id_X \otimes c_{Y,Z})$
hold for all $X,Y,Z \in \textnormal{Ob}(Z)$. A \textit{braided monoidal category} is a monoidal category equipped with a braiding $c$. 
\end{definition}

\begin{definition}\label{defn:brMonFunc}
Let $(\C,\otimes _{\C},\one_{\C},c^{\C})$ and $(\D,\otimes_{\D}, \one_{\D},c^{\D})$ be braided monoidal categories. A monoidal functor $(F, F_2, F_0):\C \rightarrow \D$ is called \textit{braided} if it satisfies
\begin{equation*}
F_2(Y,X)\; c^{\D}_{F(X),F(Y)} = F(c^{\C}_{X,Y}) \; F_2(X,Y),
\end{equation*}
for all $X,Y \in \C$.
\end{definition}

\begin{definition}\label{defn:brComAlg}
Let $(\C,\otimes,\one,c)$ be a braided monoidal category. An algebra $(A,m,u)$ in $\C$ is called \textit{(braided) commutative} if it satisfies: $m\; c_{A,A} = m$.
\end{definition}

\begin{proposition} \label{prop:commalg}
Let $(F,F_2,F_0):\C \to \D$ be a braided monoidal functor. Take $(A,m_A,u_A)$ a commutative algebra in $\C$,
then $(F(A),  F(m_A)F_2(A,A), F(u_A)F_0)$ is a commutative algebra in $\D$. 
\end{proposition}

\begin{proof}
We have that 
\[
\begin{array}{rll}
\smallskip
m_{F(A)} \; c_{F(A),F(A)}^{\mathcal{D}} 
&= F(m_A) \; F_2(A,A) \; c_{F(A),F(A)}^{\mathcal{D}}
&= F(m_A) \; F(c_{A,A}^{\C}) \; F_2(A,A) \\
\smallskip
&= F(m_A \; c_{A,A}^{\C}) \; F_2(A,A)
&= F(m_A) \; F_2(A,A) \; \; = m_{F(A)}.
\end{array}
\]
Here, the first and last equalities hold by the definition of $m_{F(A)}$. Moreover, the second, third, and fourth equalities hold by $F$ being braided and preserving compositions, and by $A$ being commutative, respectively.
\end{proof}


\subsection{Filtered and graded categories} \label{sec:filt-gr}

For this subsection, recall Hypothesis~\ref{hyp:abelian} and~\ref{hyp:strict} for a monoidal category $\C$, but the material in this section applies to the more general setting of cocomplete monoidal categories. Articles that prompt the framework of this part  include \cite{Sch1996}, \cite{BD}, \cite{ArMe}, \cite[Section~3.3]{HaMi}, \cite{GwPa}, and \cite[Section~5]{GKR}.

\begin{notation}[$\textsf{Fun}(\ess,\C), \mathbb{N}_0, \ncat, \Ncat$] Consider the following notation.
\begin{enumerate}
\item Let $\ess$ denote any posetal category, that is, a category where Hom sets each contains at most one morphism and only isomorphisms are the identity maps.

\smallskip

\item Denote by $\textsf{Fun}(\ess,\C)$ the category of functors from $\ess$ to $\C$.

\smallskip

\item Let $\mathbb{N}_0$ denote the set of natural numbers including 0. 

\smallskip

\item Let $\ncat$ denote the category with objects $\mathbb{N}_0$ and with morphisms $i \rightarrow j$ only for $i,j\in \mathbb{N}_0$ where $i\leq j$. 

\smallskip

\item Let $\Ncat$ denote the category with objects $\mathbb{N}_0$ and with only identity morphisms $\id_i$ for all $i \in \mathbb{N}_0$. 
\end{enumerate}
Note that both $\ncat$ and $\Ncat$ are examples of  posetal categories $\ess$. 
\end{notation}

Next, we define a category of $\mathbb{N}_0$-filtered objects in $\C$. 

\begin{definition}[$F_X, (\mathbb{N}_0\text{-}){\sf Fil}(\C)$] 
\label{def:fil}
Consider the following terminology. 
\begin{enumerate}
\item An object $X$ in $\C$ is called an ($\mathbb{N}_0$-)\textit{filtered object} if there exists a functor $F_X$ in ${\sf Fun}(\ncat, \C)$ such that $\text{colim}_i(F_X(i)) \cong X$ in $\C$. In this case, $F_X$ is called the \textit{filtration} associated to $X$. 

\smallskip

\item A morphism $f:X\rightarrow Y$ between filtered objects $(X, F_X)$ and $(Y,F_Y)$ in $\C$ is called an ($\mathbb{N}_0$-)\textit{filtered morphism} if there exists a natural transformation $$F_f=\{ F_f(i):F_X(i) \rightarrow F_Y(i) \}_{i\in \mathbb{N}_0}$$ in $\textsf{Fun}(\ncat,\C)$ such that $\text{colim}_i(F_f(i))=f$.

\smallskip

\item Filtered objects in $\C$ and their morphisms  form a category, denoted by ($\mathbb{N}_0$-)$\textsf{Fil}(\C)$. 

\smallskip

\item We say that a filtration on a filtered object $X$ is \textit{finite} if there exists $n \in \mathbb{N}_0$ such $F_X(i)\cong F_X(n)$ in $\C$ for all $i \geq n$. In this case, $X\cong F_X(n)$ as objects in $\C$.

\smallskip

\item We say that a filtration on a filtered object $X$ is \textit{monic} if the morphism $F_X(i \rightarrow i+1)$ is monic for each $i \in \mathbb{N}_0$.
\end{enumerate}
\end{definition}

From the definition above, we set the notation below.

\begin{notation}[$\iota^X_i$, $\psi^X_i$] \label{not:incl} Take a filtered object $(X,F_X) \in \textsf{Fil}(\C)$. 
\begin{enumerate}
    \item Let $\iota^X_i$ denote the morphism $F_X(i \rightarrow i+1)$ in $\C$, for each $i \in \mathbb{N}_0$.
    
    \smallskip
    
    \item Let $\psi^X_i: F_X(i)  \rightarrow X$ denote the canonical map derived from $\text{colim}_i(F_X(i)) \cong X$, i.e., $\psi^X_i = \psi^X_{i+1}\; \iota^X_i$ for all $i \in \ncat$.
\end{enumerate}
\end{notation}

Next, towards defining a category of graded objects in $\C$, consider the following notation.

\begin{notation}[$\C^{\mathbb{N}_0}$] Let $\C^{\mathbb{N}_0}$ be the category of sequences of objects $(X_i)_{i\in \mathbb{N}_0}$ in $\C$, and with morphisms being $\mathbb{N}_0$-graded sequences of morphisms in $\C$. Observe that we have a functor $\textstyle \C^{\mathbb{N}_0} \rightarrow \text{Ind}(\C),  (X_i)_{i\in \mathbb{N}_0} \mapsto \textstyle  \coprod_{i \in \mathbb{N}_0} X_i,$ where $\text{Ind}(\C)$ denotes the Ind-completion of $\C$. 
\end{notation}

\begin{definition}[$(\mathbb{N}_0\text{-})\textsf{Gr}(\C)$] 
\label{def:gr}
Consider the following terminology.
\begin{enumerate}
\item An object $X$ in $\C$ is called an ($\mathbb{N}_0$-)\textit{graded object} if there exists an object $(X_i)_{i \in \mathbb{N}_0}$ in $\C^{\mathbb{N}_0}$ such that $X=\coprod_{i \in \mathbb{N}_0} X_i$. 

\medskip

\item A morphism $f:X\rightarrow Y$ between graded objects $X = \coprod_{i \in \mathbb{N}_0} X_i$ and $Y = \coprod_{i \in \mathbb{N}_0} Y_i$ is called an ($\mathbb{N}_0$-)\textit{graded morphism} if there exists a morphism $(f_i:X_i \rightarrow Y_i)_{i \in \mathbb{N}_0}$ in $\C^{\mathbb{N}_0}$ such that $f=\coprod_{i\in \mathbb{N}_0} f_i$.

\medskip

\item Graded objects  in $\C$ and their morphisms  form a subcategory of $\C$, which we denote by ($\mathbb{N}_0$-)$\textsf{Gr}(\C)$.
\end{enumerate}
\end{definition}

\begin{remark} \label{rem:Ncat} 
\begin{enumerate}
    \item One can also define the category $\mathbb{N}_0$-$\textsf{Gr}(\C)$ in a manner similar to $\mathbb{N}_0$-${\sf Fil}(\C)$ by replacing the posetal category $\ncat$ by $\Ncat$.

\smallskip

\item Moreover, $\mathbb{N}_0$ could be replaced with any monoid $G$ to form the posetal category $\ess = \underline{G}$  in order to define $G$-${\sf Gr}(\C)$ as in Definition~\ref{def:gr}. Here, $\underline{G}$ denotes the category with objects being elements of $G$ and with only identity morphisms $\id_g$ for each $g \in G$. If, further, we have that $G$ is a poset, then we can define the category $G$-${\sf Fil}(\C)$ in the same manner as Definition~\ref{def:fil}. 
    \end{enumerate} 
\end{remark}

Now, we show that the categories constructed above are monoidal. This result is known in the literature; see the references listed at the beginning of this section for details. We will only provide sketches of proofs here. 

\begin{proposition} \label{prop:moncats}
Let $\C$ be a (braided) monoidal category. Then, each of the categories
\begin{center}
\textnormal{(a)} $\ncat$, 
\quad \quad \quad \quad \textnormal{(b)} ${\sf Fun}(\ncat,\C)$, 
\quad \quad \quad \quad \textnormal{(c)} ${\sf Fil}(\C)$, \quad \quad \quad \quad \textnormal{(d)} ${\sf Gr}(\C)$
\end{center}
admit the structure of a (braided) monoidal category.
\end{proposition}

\begin{proof}
(a) 
We can endow the category $\ncat$ with 
monoidal structure by defining $i \otimes j := i+j$, with unit object is $\mathbbm{1}_{\ncat} = 0$. The braiding map is $c_{i,j}=\id_{i+j}$.

\smallskip

(b) Using Day convolution \cite[page~29]{Day},
we can endow the category $\textsf{Fun}(\ncat,\C)$ with a monoidal structure as follows: given $F, G \in \textsf{Fun}(\ncat,\C)$, define 
\begin{equation*}
    (F \otimes G) (k) := \text{colim}_{i+j\leq k} (F(i) \otimes G(j));
\end{equation*}
 here, the unit object is the functor from $\ncat$ to the trivial monoidal subcategory $\one_{\C}$. 

When $\C$ is braided, we have maps $c_{F(i),G(j)}:F(i) \otimes G(j) \rightarrow G(j) \otimes F(i)$. Thus, by universal property of colimits, we get maps $(F\otimes G)(k) \rightarrow (G \otimes F)(k)$ to form a braiding on  
$\textsf{Fun}(\ncat,\C)$ making it a braided monoidal category. (See  \cite[Section~2, Example~5]{joyal1986braided}).
\smallskip

(c)  Using the monoidal structure on $\textsf{Fun}(\ncat,\C)$ in part (b), we can endow $\textsf{Fil}(\C)$ with a monoidal structure as follows. For $(X,F_X), (Y,F_Y) \in \text{Ob({\sf Fil}(}\C))$, define  
\[ 
\qquad (X,F_X) \otimes (Y,F_Y) = (X\otimes Y, F_{X\otimes Y}), \qquad \text{where }
\]
\begin{equation}\label{eq:tens-filt}
F_{X\otimes Y}(k):=\text{colim}_{i+j\leq k} F_X(i)\otimes F_Y(j).
\end{equation}
One can check that colim${}_{i}F_{X\otimes Y}(i)$ is isomorphic to $X\otimes Y$, thus the above definition is well defined.
The unit object is $\mathbbm{1}$ with the associated filtration $F_{\mathbbm{1}}: \mathbb{N}_0 \rightarrow \C$ given by $F_{\mathbbm{1}}(i)=\mathbbm{1}$ for all $i \in \mathbb{N}_0$.

Now suppose that $\C$ is braided. Using the braiding on $\textsf{Fun}(\ncat,\C)$ in part (b), we get maps, for each $k \in \mathbb{N}_0$,
\begin{equation}\label{eq:tau}
    \tau_{X,Y}(k) : F_{X\otimes Y}(k) \rightarrow F_{Y\otimes X}(k).
\end{equation}
One can check that 
\begin{equation} \label{eq:tau-c}
    \text{colim${}_k\tau_{X,Y}(k) = c_{X,Y}$.}
\end{equation} 
Furthermore, 
\begin{equation}\label{eq:tau-iota}
\tau_{X,Y}(i)\; \iota_{i-1}^{X\otimes Y} = \iota^{Y\otimes X}_{i-1} \; \tau_{X,Y}(i-1) .
\end{equation}
 Hence, $c_{X,Y}$ is a filtered map. Using part (b), we can conclude that ${\sf Fil}(\C)$ is braided with the braiding $\tau_{X,Y}$.

\smallskip

(d)  Since $\C$ is monoidal, we get that $\C^{\mathbb{N}_0}$ is monoidal, with monoidal structure $$((X_i)_{i\in \mathbb{N}_0} \otimes (Y_j)_{j \in \mathbb{N}_0})_k :=  \textstyle \coprod_{i+j=k}(X_i \otimes Y_j).$$ The monoidal unit is $ (e_i)_{i\in \mathbb{N}_0}$ with $e_i=\delta_{i,0} \mathbbm{1}$. 
Using this we can endow $\textsf{Gr}(\C)$ with a monoidal structure  where $\otimes$ inherited from $\C^{\mathbb{N}_0}$, and where unit object $\mathbbm{1}$ is $ (e_i)_{i\in \mathbb{N}_0}$ where $e_i=\delta_{i,0} \mathbbm{1}$.

When $\C$ is braided, we can collect the maps $c_{X_i,Y_j}: X_i\otimes Y_j \rightarrow Y_j \otimes X_i $ to get the maps 
\[ 
c_{X,Y}(k):= \oplus_{i+j=k}c_{X_i,Y_j}: (X\otimes Y)_k \rightarrow (Y\otimes X)_k.
\]
\noindent
Here, $c_{X,Y}$ is a braiding on $\textsf{Gr}(\C)$, making it a braided category.
\end{proof}

We will also need the following result.

\begin{corollary} \label{cor:gr-fil-epic}
The monoidal categories ${\sf Fil}(\C)$ and ${\sf Gr}(\C)$ satisfy the conditions of Hypotheses~\ref{hyp:abelian} and~\ref{hyp:strict}.
\end{corollary}

\begin{proof}
It is clear by the proof of the proposition above that ${\sf Fil}(\C)$ and ${\sf Gr}(\C)$ are strict as $\C$ is strict. Moreover, ${\sf Gr}(\C)$ satisfies the rest of Hypotheses~\ref{hyp:abelian} and~\ref{hyp:strict} component-wise due to $\C$ satisfying these  hypotheses. Moreover, $\ncat$ is small and Hypothesis~\ref{hyp:abelian} holds for $\C$, so we get that Hypothesis~\ref{hyp:abelian} also holds for ${\sf Fil}(\C) =  \textsf{Fun}(\ncat,\C)$. 

Now it suffices to show that the tensor product of ${\sf Fil}(\C)$ is biexact.
We will show that for any object $(W,F_W)$ in \textsf{Fil}$(\C)$, we obtain that $((W,F_W)\otimes -):$\textsf{Fil}$(\C) \rightarrow$ \textsf{Fil}$(\C)$ is an exact functor. A similar argument would      show that $-\otimes (W,F_W)$ is also exact. Consider any exact sequence 
\begin{equation}\label{eq:ses1}
   0 \rightarrow (X,F_X) \xrightarrow{f} (Y,F_Y) \xrightarrow{g} (Z,F_Z)\rightarrow 0
\end{equation}
in \textsf{Fil}$(\C)$. Exactness of \eqref{eq:ses1} implies that 
$ 0\rightarrow F_X(j) \rightarrow F_Y(j) \rightarrow F_Z(j)\rightarrow 0$
is exact for all $i\in \mathbb{N}$. 
Since the tensor product of $\C$ is biexact, we get that
\begin{equation}\label{eq:ses2}
0\rightarrow  F_W(i)\otimes F_X(j) \xrightarrow{\textnormal{id}\otimes {F_f(j)}} F_W(i) \otimes F_Y(j) \xrightarrow{\textnormal{id}\otimes {F_g(j)}} F_W(i)\otimes F_Z(j)\rightarrow 0    
\end{equation}
is exact for all $i,j\in \mathbb{N}$. Furthermore, the following diagrams commute,
\[
\xymatrix{
F_W(i)\otimes F_X(j) \ar[r]^(.47){\id \otimes \iota^X_j } \ar[d]_{\id \otimes F_f(j)} 
& F_W(i)\otimes F_X(j+1) \ar[d]^{\id \otimes F_f(i+1)} 
& F_W(i)\otimes F_X(j) \ar[r]^(.47){\iota^W_j \otimes \id} \ar[d]_{\id \otimes F_f(j)} 
& F_W(i+1)\otimes F_X(j) \ar[d]^{\id\otimes F_f(i)} 
\\
F_W(i)\otimes F_Y(j) \ar[r]^(.47){\id \otimes \iota^Y_{j}} 
& F_W(i)\otimes F_Y(j+1) 
& F_W(i)\otimes F_Y(j) \ar[r]^(.47){\iota^W_{j} \otimes \id} 
& F_W(i+1)\otimes F_Y(j),
}
\]
where the first diagram commutes because $f$ is a filtered map and the second commutes trivially.
We also have similar diagrams with $X,Y$ replacing $Y,Z$, respectively for all $i,j\in \mathbb{N}_0$. Recall from Definition \ref{eq:tens-filt} that $F_{W\otimes X}(k) = \text{colim}_{i+j\leq k} F_W(i) \otimes F_X(j)$. The above commutative diagrams together with the exactness of \eqref{eq:ses2} imply the exactness of:
\begin{equation*}
    0 \rightarrow F_{W\otimes X}(k) \rightarrow F_{W\otimes Y}(k) \rightarrow F_{W\otimes Z}(k)\rightarrow 0
\end{equation*}
for all $k\in \mathbb{N}_0$. Thus, the sequence
$$ 0 \rightarrow F_{W\otimes X} \xrightarrow{\id\otimes f} F_{W\otimes Y} \xrightarrow{\id \otimes g} F_{W\otimes Z}\rightarrow 0 $$
is exact. 
\end{proof}

Now using Proposition~\ref{prop:moncats}, we set the following terminology.

\begin{definition} \label{def:fil-gr-C}
Take an algebraic structure $X$ in $\C$, e.g., $X$ is either an algebra, a coalgebra, a Frobenius algebra, a left/right module, a left/right weak ideal, or a left/right ideal in $\C$. If $X$ belongs to the monoidal category $\textsf{Fil}(\C)$, then we say that $X$  is a \textit{filtered  structure} in $\C$. Likewise, if $X$  is in $\textsf{Gr}(\C)$, then $X$ is called a  {\it graded structure} in $\C$.
\end{definition}

Next, we study when a left ideal of a filtered algebra in $\C$ admits the structure of a filtered left ideal of $A$. Consider the preliminary results below.

\begin{lemma}\label{lem:filMor}
Consider three filtered objects $(X,F_X),(Y,F_Y),(Z,F_Z)$ in $\C$. Then the following statements are equivalent.
\begin{enumerate}[(a),font=\upshape]
    \item There exists a filtered map $f: X \otimes Y \rightarrow Z \in \textnormal{\textsf{Fil}}(\C)$.
    
    \smallskip
   
    \item There exist morphisms $f_{i,j}: F_X(i) \otimes F_Y(j) \rightarrow F_Z(i+j) \in \C$  for $i,j \in \mathbb{N}_0$ such that the following diagrams commute:
    \begin{small}
    \begin{center}
    \begin{tikzcd}
    F_X(i)\otimes F_Y(j) \arrow[d, "{f_{i,j}}"'] \arrow[r, " \iota^X_i 
    \otimes \textnormal{id}"] &  F_X(i+1)\otimes F_Y(j) \arrow[d, "{f_{i+1,j}}"] & F_X(i)\otimes F_Y(j) \arrow[d, "{f_{i,j}}"'] \arrow[r, "\textnormal{id}\otimes \iota^Y_j"] &  F_X(i)\otimes F_Y(j+1) \arrow[d, "{f_{i,j+1}}"] \\
    F_Z(i+j) \arrow[r, "\iota^Z_{i+j}"']  & F_Z(i+j+1)   & F_Z(i+j) \arrow[r, "\iota^Z_{i+j}"']   & F_Z(i+j+1).                   
    \end{tikzcd}    
    \qed  
    \end{center}
    \end{small}
\end{enumerate}
\end{lemma}

\begin{lemma}\label{lem:induced}
Take $(Y,F_Y) \in \textnormal{\textsf{Fil}}(\C)$ with finite filtration, and let $f:X\rightarrow Y $ be a morphism in $\C$. Then there exists a filtration $F_X$ on $X$ such that $f \in \textnormal{\textsf{Fil}}(\C)$.
\end{lemma}

\begin{proof}
Since $Y$ has a finite filtration, there exists $n\in \mathbb{N}_0$ such that $F_Y(n+k)=F_Y(n)$ for all $k\in \mathbb{N}_0$. Hence, colim${}_i F_Y(i)=F_Y(n)=Y$.  Define $F_X(n) = X$, and define $F_X(n-1)$ to be pullback of the morphisms $f: F_X(n) \to F_Y(n)$ and $\iota_{n-1}^Y:F_Y(n-1) \to F_Y(n)$ via morphisms $p_{n-1}:F_X(n-1) \to F_X(n)$ and $f_{n-1}: F_X(n-1) \to F_Y(n-1)$. Likewise, inductively, define $F_X(i)$ to be pullback of the morphisms $f_{i+1}: F_X(i+1) \to F_Y(i+1)$ and $\iota_{i}^Y:F_Y(i) \to F_Y(i+1)$ via morphisms $p_{i}:F_X(i) \to F_X(i+1)$ and $f_{i}: F_X(i) \to F_Y(i)$. Thus, we get the composition of commutative diagrams below:
\begin{small}
\[ 
\xymatrix@R=2pc@C=1.8pc{
F_X(0)\ar[d]_{f_0} \ar[r]^{p_0} \pullbackcorner & F_X(1)  \ar[d]_{f_1} \ar[r]^{\quad p_1} \pullbackcorner & \cdots \ar[r]^{p_{i-1}\;\;}  & F_X(i)\ar[d]_{f_i} \ar[r]^{p_i \quad} \pullbackcorner & F_X(i+1)\ar[d]_{f_{i+1}}\ar[r]^{\quad p_{i+1}} & \cdots \ar[r]^{p_{n-2}\quad \quad}  & F_X(n-1)\ar[d]_{f_{n-1}} \ar[r]^{\quad \quad p_{n-1}} \pullbackcorner & X\ar[d]^{f}\\
F_Y(0) \ar[r]_{\iota^Y_0} & F_Y(1) \ar[r]_{\iota^Y_1}  & \cdots\ar[r]_{\iota^Y_{i-1} \quad} & F_Y(i)\ar[r]_{\iota^Y_{i} \quad} & F_Y(i+1)\ar[r]_{\quad \iota^Y_{i+1}}  & \cdots \ar[r]_{\iota^Y_{n-2} \quad \quad} & F_Y(n-1) \ar[r]_{\quad \quad \iota^Y_{n-1}}& Y
}
\]
\end{small}

\noindent Since $\text{colim}_i F_X(i) = F_X(n)=X$, we get that $X$ is a filtered object with filtration $F_X$, and moreover, $f$ is a filtered map via $f_i$. Namely, in the diagram above, we have 
$$f_i = F_f(i) \quad \quad \text{and} \quad \quad p_i = \iota^X_{i}.$$

\vspace{-.25in}

\end{proof}

\begin{proposition}\label{prop:idealFil}
Let $(I,\lambda,\phi)$ be left ideal of a filtered algebra $(A,F_A,m,u)$ with finite filtration. Then, with the induced filtration on $I$ via $\phi$ \textnormal{[Lemma \ref{lem:induced}]}, we obtain that $(I,\lambda,\phi)$ is filtered left ideal of $A$. 
\end{proposition}

\begin{proof}
Since $A$ has a finite filtration, there exists $n\in \mathbb{N}$ such that $F_A(n+k)=F_A(n)$ for all $k\in \mathbb{N}$.
Let $F_I$ denote the induced filtration on $I$ from Lemma~\ref{lem:induced}, and let $$(\phi,F_{\phi}): (I,F_I)\rightarrow (A,F_A)$$ be the filtered map between $I$ and $A$.
Furthermore, since $\phi$ is a mono, $F_{\phi}(i)$ is a mono for all $i$ as pullbacks preserve monomorphisms. 

Since $(A,F_A)$ is a filtered algebra with multiplication $m: A \otimes A \to A$ (that is associative),  we have   maps $m_{i,j}: F_A(i)\otimes F_A(j)\rightarrow F_A(i+j)$ in $\C$ satisfying
\begin{equation}\label{eq:filAlg}
m_{i+j,k} (m_{i,j}  \otimes \id_{F_A(k)}) = m_{i,j+k}(\id_{F_A(i)} \otimes m_{j,k})
\end{equation}
 by Lemma~\ref{lem:filMor}. To show that $(I,F_I)$ is a filtered left ideal of $A$, we need to first give a filtration for the maps $\lambda: A\otimes I \rightarrow I$ and  $\phi: I \rightarrow A$. We have already done so for $\phi$  in Lemma \ref{lem:induced}. To do the same for $\lambda$, by Lemma \ref{lem:filMor}, it suffices to construct maps $$\lambda_{i,j}:F_A(i)\otimes F_I(j) \rightarrow F_I(i+j),$$ and check that:
    \begin{align}\label{eq:filLambda}
    \begin{split}
        \lambda_{i+1,j}\; (\iota^A_i\otimes \text{id}_{F_I(j)}) =  \iota^I_{i+j}\;\lambda_{i,j} \quad \text{ and } \quad
        \lambda_{i,j+1}\; (\text{id}_{F_A(i)} \otimes \iota^I_j) =  \iota^I_{i+j}\; \lambda_{i,j}.    
    \end{split}
    \end{align}
We also need to check that $\lambda_{i,j}$ makes $I$ a filtered left $A$-module, that is,
    \begin{equation}\label{eq:filModule1}
        \lambda_{i+j,k}\; (m_{i,j} \otimes \text{id}_{F_{I}(k)}) = \lambda_{i,j+k}\; (\text{id}_{F_A(i)} \otimes  \lambda_{j,k}), 
    \end{equation}
and that $\phi$ is a map of filtered left $A$-modules, that is,
    \begin{equation}\label{eq:filModule2}
        F_{\phi}(i+j) \; \lambda_{i,j} = m_{i,j} \; (\text{id}_{F_A(i)} \otimes F_{\phi}(j)). 
    \end{equation}

The proofs of (\ref{eq:filLambda}) and (\ref{eq:filModule1}) use the same idea, so  we will only discuss \eqref{eq:filModule1} and \eqref{eq:filModule2}, and leave the verification of \eqref{eq:filLambda} to the reader. Consider the following equations:
\begin{align*}
\begin{split}
    \phi\; \lambda\; (\psi^A_i\otimes \psi^I_{j}) 
    & =  m\; (\text{id}_A\otimes \phi)\; (\psi^A_i \otimes \psi^I_{j})\\
    & = m\; (\psi^A_i\otimes \psi^A_j) \; (\text{id}_{F_A(i)} \otimes F_{\phi}(j) ) \\
    &= \psi^A_{i+j}\; m_{i,j} \; (\text{id}_{F_A(i)} \otimes F_{\phi}(j) ).
\end{split} 
\end{align*}
Here, the first equality holds by axioms of $I$ being a left ideal of $A$; the second equality follows from the diagram  in Lemma~\ref{lem:induced} and Notation~\ref{not:incl}; and the third equality holds because $m$ is a filtered map (see, again, Notation~\ref{not:incl}). Hence, the outer square of the following diagram commutes.
Moreover, we can use the universal property of pullbacks to get maps $\lambda_{i,j}$ such that diagrams $(\RN{1})$ and $(\RN{2})$ below commute as well.
\[
\xymatrix@R=.9pc@C=3pc{
F_A(i) \otimes F_I(j) \ar[dd]_{\text{id}_{F_A(i)} \otimes F_{\phi}(j)}  \ar[r]^{\quad\;\; \psi^A_i \otimes \psi^I_j} \ar@{.>}[rd]_{\lambda_{i,j}}   & A\otimes I \ar[rrd]^{\lambda} \ar@{}[d]|(0.5){(\RN{2})}  \\
& F_I(i+j) \pullbackcorner \ar[rr]_{\psi^I_{i+j}} \ar[dd]^{F_{\phi}(i+j)} & & I \ar[dd]^{\phi} \\
F_A(i)\otimes F_A(j) \ar[rd]_{m_{i,j}} \ar@{}[ur]|(0.5){(\RN{1})}  & & & \\
&  F_A(i+j) \ar[rr]_{\quad \psi^A_{i+j}}  & & A \\
}
\]
Commutation of $(\RN{1})$ implies (\ref{eq:filModule2}). Commutation of $(\RN{2}$) implies that $\lambda_{i,j}$ are a filtration of $\lambda$. Moreover, we can now conclude that
\begin{align*}
F_{\phi}(i+j+k) \; \lambda_{i+j,k}\; (m_{i,j}\otimes \text{id}_{F_I(k)}) 
& \stackrel{(\RN{1})}{=} m_{i+j,k} \; (\text{id}_{F_A(i+j)} \otimes F_{\phi}(k))  \; (m_{i,j}\otimes \text{id}_{F_I(k)})  \\
& = m_{i+j,k} \; (m_{i,j} \otimes \text{id}_{F_A(k)}) \; (\text{id}_{F_A(i)\otimes F_A(j)} \otimes F_{\phi}(k)) \\
& \stackrel{\textnormal{(\ref{eq:filAlg})}}{=} m_{i,j+k} \; (\text{id}_{F_A(i)} \otimes m_{j,k}) \; (\text{id}_{F_A(i)\otimes F_A(j)} \otimes F_{\phi}(k)) \\
& \stackrel{(\RN{1})}{=} m_{i,j+k} \; (\text{id}_{F_A(i)} \otimes F_{\phi}(j+k)) \; (\text{id}_{F_A(i)} \otimes \lambda_{j,k}) \\
& \stackrel{(\RN{1})}{=} F_{\phi}(i+j+k) \; \lambda_{i,j+k} \; (\text{id}_{F_A(i)} \otimes \lambda_{j,k}). 
\end{align*}
Since, $F_{\phi}(i+j+k)$ is a monomorphism, we get that (\ref{eq:filModule1}) holds, as desired.
\end{proof}


\section{Associated graded constructions for monoidal categories}  \label{sec:assgr}

In this part, we present the first main result of this paper: the associated graded construction for filtered algebras in monoidal categories [Theorem~\ref{thm:gr}]. Compare to \cite[Section~3]{ArMe} and \cite[Sections~5.2.3,~5.3.3]{GKR}. Recall Hypotheses~\ref{hyp:abelian} and~\ref{hyp:strict} for a monoidal category $(\C, \otimes, \unit)$.  In addition to the notation of Section~\ref{sec:monoidal}, consider the notation below.

\begin{notation}[$\overline{F_A(i)},  \pi^A$, $\beta^{A,B}_{i,j}$] \label{not:filt}
Take filtered objects $(A,F_A), (B,F_B) \in \textsf{Fil}(\C)$.
\begin{enumerate}
    \item Let $\overline{F_A(i)}$  denote coker($F_A(i-1) \xrightarrow{\iota^A_{i-1}} F_A(i) $).
    
    \smallskip
    
    \item Let $\pi^A_i$ denote the canonical epimorphism $F_A(i) \to \overline{F_A(i)}$, for each $i \in \mathbb{N}_0$. (By convention, $\overline{F_A(0)} = F_A(0)$, so that $\pi_0^A = \id_{F_A(0)}$.)

\smallskip

\item Define the map 
$$\beta_{i,j}^{A,B}: F_A(i) \otimes F_B(j) \rightarrow  F_{A\otimes B}(i+j),$$
to be the natural map to the colimit as in \eqref{eq:tens-filt}. 
\end{enumerate}
\end{notation}

Then by properties of colimits, we have that:
\begin{equation}\label{eq:AssoCo}
\beta^{A\otimes B,C}_{i+j,k}\; (\beta^{A,B}_{i,j}\otimes \id_{F_C(k)}) = \beta_{i,j+k}^{A,B\otimes C} \; (\id_{F_A(i)} \otimes \beta^{B,C}_{j,k}).      
\end{equation}

In the case when $\C$ is braided, we will also require the notation and identities below.

\begin{notation}[$\tau_{A,B}(k)$,  $\overline{\tau_{A,B}(k)}$]
\label{not:tau}
Let $(\C, \otimes, \unit, c)$ be a braided monoidal category, and take filtered objects $(A,F_A), (B,F_B), (C,F_C)$ of $\C$. Recall from \eqref{eq:tau}, \eqref{eq:tau-c}, \eqref{eq:tau-iota} the component 
$$\tau_{A,B}(k):F_{A\otimes B}(k) \to F_{B\otimes A}(k)$$ of the filtration on the braiding map $c_{A,B}$. 
 Observe that by \eqref{eq:tau-c} we have:
\begin{equation}\label{eq:brCoker1}
  \tau_{X,Y}(i+j)\; \beta^{X,Y}_{i,j} = \beta^{Y,X}_{j,i} \; c_{F_X(i),F_Y(j)} ,   
\end{equation}
\noindent for all $i, j \in \mathbb{N}_0$ and $X,Y \in \C$.
By universal property of cokernels, for each $k \in \mathbb{N}_0$ and $X,Y \in \C$, there exists a morphism $\overline{\tau_{X,Y}(k)}$ such that the following diagram commutes
\begin{center}
\begin{tikzcd}
F_{X\otimes Y}(k-1) \arrow[r, "\iota^{X\otimes Y}_{k-1}"] \arrow[d, "{\tau_{X,Y}(k-1)}"'] & F_{X\otimes Y}(k) \arrow[r, "\pi^{X\otimes Y}_{k}"] \arrow[d, "{\tau_{X,Y}(k)}"] & \overline{F_{X\otimes Y}(k)} \arrow[d, "{\overline{\tau_{X,Y}(k)}}", dotted] \\
F_{Y\otimes X}(k-1) \arrow[r,"\iota^{Y\otimes X}_{k-1}"]                                        & F_{Y\otimes X}(k) \arrow[r,"\pi^{Y\otimes X}_{k}"]                                                  & \overline{F_{Y\otimes X}(k)}.                  
\end{tikzcd}
\end{center}
Thus we get that
\begin{equation}\label{eq:brCoker2}
    \overline{\tau_{X,Y}(k)}\; \pi^{X\otimes Y}_k
 = \pi^{Y\otimes X}_k \; \tau_{X,Y}(k).
 \end{equation}
\end{notation}

\smallskip

Next, we present the construction of the associated graded functor.

\begin{definition}[gr] \label{def:gr-functor}
We define the {\it associated graded functor}, $$\text{gr}: \textsf{Fil}(\C) \rightarrow \textsf{Gr}(\C),$$ as follows. For an object $(A,F_A) \in  \textsf{Fil}(\C)$, let
$$\text{gr}(A, F_A) = \textstyle \coprod_{i \in \mathbb{N}_0} \overline{F_A(i)}.$$ 
Given a morphism $f:(A,F_A) \rightarrow (B,F_B)$ in $\textsf{Fil}(\C)$, define 
$\gr(f) : \text{gr}(A,F_A) \rightarrow \text{gr}(B,F_B)$ with components coming from universal property of cokernels:
\begin{equation} \label{eq:gr-comp}
\xymatrix{
F_A(i-1) \ar[r]^(.55){\iota^A_{i-1}} \ar[d]_{F_f(i-1)} & F_A(i) \ar[r]^{\pi^A_{i}} \ar[d]^{F_f(i)}          & \overline{F_A(i)} \ar@{..>}[d]^{\gr(f)_i} \ar[r] & 0 \\
F_B(i-1) \ar[r]^(.55){\iota^B_{i-1}}                       & F_B(i) \ar[r]^{\pi^B_{i}}    & \overline{F_B(i)} \ar[r] & 0.             
}
\end{equation}
\end{definition}

Now we come to the main result of this section.

\begin{theorem} \label{thm:gr}
If $\C$ a (braided) monoidal category, then $\gr: \textnormal{\textsf{Fil}}(\C) \rightarrow \textnormal{\textsf{Gr}}(\C)$ is a (braided) monoidal functor.
\end{theorem}

\begin{proof}
To show that the functor $\gr$  is  monoidal, we will:
\begin{enumerate}
\item[(i)] construct a natural transformation $\gr_2: \gr(-) \otimes \gr(-) \to \gr(- \otimes -)$ that satisfies the associativity condition in Definition~\ref{def:monfunc};
\item[(ii)] construct a morphism $\gr_0: \mathbbm{1}_{\textsf{Gr}(\C)} \rightarrow \gr(\mathbbm{1}_{\textsf{Fil}(\C)}) \in \textsf{Gr}(\C)$ that satisfies the unitality condition in Definition~\ref{def:monfunc}; and will
\item[(iii)] verify that $\gr$ is braided (Definition~\ref{defn:brMon}) when $\C$ is braided.
\end{enumerate}

\medskip

(i) Take objects $A:=(A,F_A)$ and $B:=(B, F_B)$ in ${\sf Fil}(\C)$. To define a natural transformation $\gr_2=\{ \gr_2(A,B): \gr(A)\otimes \gr(B) \rightarrow \gr(A\otimes B)  \}$, it suffices to define componentwise maps $$\Theta^{A,B}_{i,j}: \overline{F_A(i)} \otimes \overline{F_B(j)}\rightarrow \overline{ F_{A\otimes B}(i+j)}$$ and check the associativity condition for these maps for all $i, j \in \mathbb{N}_0$.
To proceed, consider the exact sequence
$F_A(i-1) \xrightarrow{\iota^A_{i-1}} F_A(i) \xrightarrow{\pi^A_{i}}  \overline{F_A(i)} \rightarrow 0$. 
Apply the right exact functor $-\otimes F_B(j)$ [Hypothesis~\ref{hyp:strict}] to this sequence to get the following exact sequence.
\begin{equation}\label{eq:cokerAB}
F_A(i-1)\otimes F_B(j) \xrightarrow{\iota^A_{i-1} \otimes \textnormal{id}}  F_A(i)\otimes F_B(j) \xrightarrow{\pi^A_{i} \otimes \textnormal{id}}  \overline{F_A(i)} \otimes F_B(j) \rightarrow 0.
\end{equation}

On the other hand, recall Notation~\ref{not:incl}, and note that properties of colimits yield
\begin{equation}\label{eq:coker}
     \iota^{A\otimes B}_{i+j-1} \; \beta^{A,B}_{i-1,j} = \beta^{A,B}_{i,j} \; (\iota^A_{i-1} \otimes \id_{ F_B(j)}).
\end{equation}

\smallskip 
\noindent
 By definition of cokernels, we know that $\pi^{A\otimes B}_{i+j} \; \iota^{A\otimes B}_{i+j-1} = 0 $.
So, using (\ref{eq:coker}), we get that 
\begin{equation}\label{eq:theta}
\pi^{A\otimes B}_{i+j} \; \beta^{A,B}_{i,j} \; (\iota^A_{i-1} \otimes \id_{ F_B(j)})  = (\pi^{A\otimes B}_{i+j} \; \iota^{A\otimes B}_{i+j-1}) \; \beta^{A,B}_{i-1,j}  =0.
\end{equation}
As $-\otimes F_B(j)$ is exact, coker$(\iota^A_{i-1} \otimes \textnormal{id}) = \pi^A_{i} \otimes \textnormal{id}$, we can use \eqref{eq:cokerAB} and (\ref{eq:theta}) and the universal property of cokernels to get the map $\theta_{i,j}^{A,B}$ such that the following diagram ($\dagger$) commutes: 
\[
\xymatrix@C=3pc{
F_A(i-1)\otimes F_B(j)  \ar[r]^{\quad \iota^A_{i-1} \otimes \textnormal{id}}  & F_A(i)\otimes F_B(j) \ar@{}[dr]|(.5){(\dagger)} \ar[d]_{\beta_{i,j}^{A,B}}  \ar[r]^{\pi^A_{i} \otimes \textnormal{id}}   & \overline{F_A(i)} \otimes F_B(j) \ar[r] \ar@{.>}[d]^{\theta_{i,j}^{A,B}} & 0 \\
  & F_{A\otimes B}(i+j) \ar[r]_{\pi^{A\otimes B}_{i+j}}  & \overline{F_{A\otimes B}(i+j)} \ar[r] & 0.
}
\]
Now consider the exact sequence $F_B(j-1) \xrightarrow{\iota^B_{i-1}} F_B(j) \xrightarrow{\pi^B_{j}} \overline{F_B(j)} \rightarrow 0 $.
Apply the functor $\overline{F_A(i)} \otimes -$ to the above sequence to get : 
\begin{equation}\label{eq:cokerB}
\overline{F_A(i)} \otimes F_B(j-1) \xrightarrow{\id \otimes \iota^B_{j-1}} \overline{F_A(i)} \otimes F_B(j) \xrightarrow{\id \otimes \pi^B_{j}} \overline{F_A(i)} \otimes \overline{F_B(j)} \rightarrow 0. 
\end{equation}

\noindent
Consider the following commutative diagram:
\[ 
\xymatrix@C=2.6pc{
& \overline{F_A(i)}\otimes F_B(j-1) \ar@{}[d]|(.5){(1)} \ar[r]^{\quad \id \otimes \iota^B_{j-1}} & \overline{F_A(i)}\otimes F_B(j) \ar[rd]^{\theta_{i,j}^{A,B}} &                                \\
F_A(i)\otimes F_B(j-1) \ar[ru]^{\pi^A_i \otimes\id} \ar[rd]_{\beta^{A,B}_{i,j-1}} \ar[r]^{\id \otimes \iota^B_{j-1}} & F_A(i) \otimes F_B(j) \ar@{}[d]|(.5){(2)} \ar[rd]^{\beta^{A,B}_{i,j}} \ar[ru]_{\pi^A_{i} \otimes \id} \ar@{}[rr]|(.5){(\dagger)}     &    & \overline{F_{A\otimes B}(i+j)} \\
& F_{A\otimes B}(i+j-1) \ar[r]_{\iota^{A\otimes B}_{i+j-1}}   & F_{A\otimes B}(i+j) \ar[ru]_{\pi^{A\otimes B}_{i+j}}.    &  
}
\]

\noindent
Here, $(1)$ commutes naturally,
and $(2)$ commutes by  \eqref{eq:coker}.
Since the composition along the lower boundary is $0$, if we move along the upper boundary of this diagram, we get that $ \theta_{i,j}^{A,B} (\id \otimes \iota^B_{j-1})  (\pi^A_i \otimes\id) = 0$.
Since $\pi^A_i \otimes \id$ is epic, we get that $\theta_{i,j}^{A,B}  (\id \otimes \iota^B_{j-1}) =0$.
Since $\coker(\id \otimes \iota_{j-1}^B) = \id \otimes \pi_{j}^B$, we can use \eqref{eq:cokerB} and the universal property of cokernels to get maps $\Theta_{i,j}^{A,B}$ satisfying the commutative diagram below:
\[
\xymatrix@C=3pc{
\overline{F_A(i)} \otimes F_B(j-1) \ar[r]^{\quad \id \otimes \iota^B_{j-1}}   & \overline{F_A(i)} \otimes F_B(j) \ar[rd]_{\theta_{i,j}^{A,B}} \ar[r] ^{\id \otimes \pi^B_j} & \overline{F_A(i)} \otimes \overline{F_B(j)} \ar@{.>}[d]^{\Theta_{i,j}^{A,B}} \ar@{}[dl]|(.35){(\ddagger)} \\
    & & \overline{F_{A\otimes B}(i+j)}. 
}
\]

Finally, combining ($\dagger$) and ($\ddagger$), we get a map $\Theta_{i,j}: \overline{F_A(i)} \otimes \overline{F_B(j)} \rightarrow  \overline{F_{A\otimes B}(i+j)}$, unique up to isomorphism,
such that the following diagram  commutes:
\[
\xymatrix@C=3pc{
F_A(i) \otimes F_B(j) \ar[r]^{\pi^A_i \otimes \pi^B_j} \ar@{}[dr]|(.5){(\Upsilon)} \ar[d]_{\beta^{A,B}_{i,j}}    & \overline{F_A(i)} \otimes \overline{F_B(j)} \ar@{.>}[d]^{\Theta_{i,j}} \\
F_{A\otimes B}(i+j) \ar[r]_{\pi^{A\otimes B}_{i+j}} & \overline{F_{A\otimes B}(i+j)} . 
}
\]

Collecting all of the morphisms $\Theta_{i,j}^{A,B}$ for $i,j \in \mathbb{N}_0$, we get the required map 
$$\gr_2(A,B)= \Theta: \gr(A)\otimes \gr(B) \rightarrow \gr(A\otimes B)  .$$
Now to verify the associativity condition, consider the following commutative diagram: 
\begin{small}
\[
\xymatrix@R=1pc{
F_A(i)\otimes F_B(j)\otimes F_C(k) \ar@{}[dr]|(1){(\Upsilon)} \ar[dd]_{\pi^A_i \otimes \pi^B_j \otimes \pi^C_k} \ar[rr]^{\beta^{A,B}_{i,j}\otimes \id_{F_C(k)}} &  & F_{A\otimes B}(i+j)\otimes F_C(k) \ar@{}[dr]|(1){(\Upsilon)} \ar[dd]^{\pi^{A\otimes B}_{i+j} \otimes \pi^C_{k}} \ar[rr]^{\beta^{A\otimes B,C}_{i+j,k}} &  & F_{A\otimes B\otimes C}(i+j+k) \ar[dd]^{\pi^{A\otimes B\otimes C}_{i+j+k}} \\                                        &  &          &  &        \\
\overline{F_A(i)} \otimes \overline{F_B(j)}\otimes \overline{F_C(k)} \ar[rr]_{\Theta^{A,B}_{i,j}\otimes \id_{\overline{F_C(k)}}}    &  & \overline{F_{A\otimes B}(i+j)}\otimes \overline{F_C(k)} \ar[rr]_{\Theta^{A\otimes B,C}_{i+j,k}}    &  & \overline{F_{A\otimes B\otimes C}(i+j+k)}.  
}
\]
\end{small}

\noindent
Thus we get that
\begin{small}
\begin{equation}\label{eq:cond1}
\Theta^{A\otimes B,C}_{i+j,k} \; (\Theta^{A,B}_{i,j}\otimes \id_{\overline{F_C(k)}}) \; (\pi^A_i \otimes \pi^B_j \otimes \pi^C_k) =  \pi^{A\otimes B\otimes C}_{i+j+k} \; {\beta^{A\otimes B,C}_{i+j,k}} \; (\beta^{A,B}_{i,j}\otimes \id_{F_C(k)}).
\end{equation}
\end{small}

\noindent
Similarly, we have that
\begin{small}
\begin{equation}\label{eq:cond2}
\Theta^{A, B\otimes C}_{i,j+k} \; (\id_{\overline{F_A(i)}} \otimes \Theta^{B,C}_{j,k}) \; (\pi^A_i \otimes \pi^B_j \otimes \pi^C_k) =  \pi^{A\otimes B\otimes C}_{i+j+k} \; {\beta^{A, B \otimes C}_{i,j+k}} \; (\id_{F_A(i)} \otimes \beta^{B,C}_{j,k}).
\end{equation}
\end{small}

\noindent
Combining the results above, we get that
\begin{align*}
& \Theta^{A\otimes B,C}_{i+j,k} \; (\Theta^{A,B}_{i,j}\otimes \id_{\overline{F_C(k)}}) \; (\pi^A_i \otimes \pi^B_j \otimes \pi^C_k)\\
& \quad \stackrel{\textnormal{(\ref{eq:cond1})}}{=}
\pi^{A\otimes B\otimes C}_{i+j+k} \; \beta^{A\otimes B,C}_{i+j,k} \; (\beta^{A,B}_{i,j} \otimes \id_{F_C(k)})    \\
 & \quad \stackrel{\textnormal{(\ref{eq:AssoCo})}}{=}  \pi^{A\otimes B\otimes C}_{i+j+k} \; {\beta^{A, B \otimes C}_{i,j+k}} \; (\id_{F_A(i)} \otimes \beta^{B,C}_{j,k}) \\
 & \quad \stackrel{\textnormal{(\ref{eq:cond2})}}{=} \Theta^{A, B\otimes C}_{i,j+k} \; (\id_{\overline{F_A(i)}} \otimes \Theta^{B,C}_{j,k}) \; (\pi^A_i \otimes \pi^B_j \otimes \pi^C_k).
\end{align*}

\noindent
Since $\pi^A_i \otimes \pi^B_j \otimes \pi^C_k$ is epic,  the associativity relation holds:
$$ \Theta^{A\otimes B,C}_{i+j,k} \; (\Theta^{A,B}_{i,j}\otimes \id_{\overline{F_C(k)}}) = \Theta^{A, B\otimes C}_{i,j+k} \; (\id_{\overline{F_A(i)}} \otimes \Theta^{B,C}_{j,k}). $$

\medskip

(ii) Recall that $\mathbbm{1}_{\textsf{Fil}(\C)} = \mathbbm{1}_\C$ with filtration $F_{\mathbbm{1}}(i) = \mathbbm{1}_\C$ for all $i\in \mathbb{N}_0$ with identity maps between the components. Thus, $\text{gr}(\mathbbm{1}_{\textsf{Fil}(\C)}) = \mathbbm{1}_\C$ with $\mathbb{N}_0$-grading $\mathbbm{1}_\C \oplus 0 \oplus 0 \ldots$. Therefore,  $\text{gr}(\mathbbm{1}_{\textsf{Fil}(\C)}) =  \mathbbm{1}_{\textsf{Gr}(\C)}$, and thus we define
$$\text{gr}_0 = \id_{\mathbbm{1}_{\textsf{Gr}(\C)}}.$$
By the commutative diagram ($\Upsilon$), one can check that $\gr_2(\one_{\textsf{Fil}(\C)},A) = \id_{\gr(A)} = \gr_2(A,\one_{\textsf{Fil}(\C)})$. Thus, $\text{gr}_0$ and $\text{gr}_2$ satisfy the unitality condition for $\text{gr}$ to be monoidal.

\medskip

(iii)
Lastly, we will check that the functor $\gr$ is braided when $\C = (\C, \otimes, \unit, c)$ is braided. Recall Notation~\ref{not:tau}, and consider the computation below:
\begin{align*}
\overline{\tau_{A,B}(i+j)}\; \Theta^{A,B}_{i,j} \; (\pi^A_i\otimes \pi^B_j ) 
& \stackrel{(\Upsilon)}{=} \overline{\tau_{A,B}(i+j)}\; \pi^{A\otimes B}_{i+j} \; \beta^{A,B}_{i,j} \\
& \stackrel{\textnormal{(\ref{eq:brCoker2})}}{=} \pi^{A\otimes B}_{i+j} \; \tau_{A,B}(i+j) \; \beta^{A,B}_{i,j} \\
& \stackrel{\textnormal{(\ref{eq:brCoker1})}}{=} \pi^{A\otimes B}_{i+j} \; \beta_{j,i}^{B,A} \; c_{F_A(i),F_B(j)} \\
& \stackrel{(\Upsilon)}{=} \Theta^{B,A}_{j,i} \; (\pi^B_j \otimes \pi^A_i) \; c_{F_A(i),F_B(j)} \\
& \stackrel{(\gamma)}{=} \Theta^{B,A}_{j,i} \; c_{\hspace{.01in} \overline{F_A(i)}, \overline{F_B(j)}} \; (\pi^A_i\otimes \pi^B_j ). 
\end{align*}
\noindent
The identity $(\gamma)$ holds by the naturality of the braiding map $c$. Since the morphism $\pi^A_i\otimes \pi^B_j$ is epic, we get that 
\begin{equation*}\label{eq:braid}
    \overline{\tau_{A,B}(i+j)}\; \Theta^{A,B}_{i,j} = \Theta^{B,A}_{j,i} \; c_{\hspace{.01in}\overline{F_A(i)}, \overline{F_B(j)}}. 
\end{equation*}
This is precisely the $k$-th component of the equation below with $k = i+j$:
\[ 
\gr(c_{X,Y}) \; \gr_2(A,B) = \gr_2(B,A) \; c_{\gr(A),\gr(B)}.
\]
Thus, gr is braided, as desired.
\end{proof}

\begin{corollary} \label{cor:gr-pres}
The functor $\gr$ sends filtered (commutative) algebras  to graded (commutative) algebras, and also sends  filtered left/right modules (resp., filtered left/right weak ideals) to  graded left/right modules (resp., graded left/right weak ideals).
\end{corollary}

\begin{proof}
This follows from Theorem~\ref{thm:gr} and Propositions~\ref{prop:mon-preserve} and~\ref{prop:commalg}.
\end{proof}

\begin{definition} \label{def:assgrstr}
If $X$ is a filtered structure (e.g., algebra, left/right module, left/right weak ideal) in {\sf Fil}($\C$), then we call gr($X$) in {\sf Gr}($\C$) the {\it associated graded structure of~$X$}, and refer to $X$ as a {\it filtered deformation} of gr($X$).
\end{definition}

Moreover, it is known that the associated graded functor is right exact; see, e.g., \cite[Remark~1.6, Lemma~3.30]{GwPa}; we include some details of the proof for the reader's convenience.s

\begin{proposition}
 \label{prop:gr-leftadj} 
The associated graded functor $\gr$ is left adjoint to the functor
$$\textnormal{triv}: {\sf Gr}(\C) \to {\sf Fil}(\C), \quad \textstyle \coprod_{i \in \mathbb{N}_0} X_i \mapsto (X_0 \overset{0}{\to} X_1 \overset{0}{\to} X_2 \overset{0}{\to}  \dots).$$
As a consequence, $\gr$ is right exact. 
\end{proposition}

\begin{proof}
For $X = \coprod_{i \in \mathbb{N}_0} X_i \in {\sf Gr}(\C)$, define the counit of the adjunction $\varepsilon: \gr \circ \textnormal{triv} \Rightarrow \id_{{\sf Gr}(\C)}$ by $(\ep_X)_i = \id_{X_i}$, that is, $\ep_X = \id_X$. Moreover, for $(Y,F_Y) \in {\sf Fil}(\C)$, define the unit of the adjunction $\eta: \id_{{\sf Fil}(\C)} \Rightarrow \textnormal{triv} \circ \gr$ by $(\eta_Y)(i) = \pi_i^Y :F_Y(i) \to \overline{F_Y(i)}$. We leave it to the reader to check the triangle axioms to get that gr $\dashv$ triv. The consequence is well-known.
\end{proof}

\begin{remark}
The results in this section also hold if we replace the monoid $\mathbb{N}_0$ by a partially ordered set that is also a monoid.
\end{remark} 


\section{Quotient algebras in monoidal categories}
\label{sec:quotient}

In this section, we discuss the construction of quotient algebras in monoidal categories. After presenting the categorical setting for this material, we expand on results from \cite{BD} to define quotient algebras as cokernels of weak ideal maps [Proposition~\ref{prop:quotalg}]. Then, we examine quotient algebras via monoidal functors [Proposition~\ref{prop:quot-functor}], especially for the associated graded functor constructed in the previous section [Corollary~\ref{cor:quot-gr}]. We discuss in Remark~\ref{rem:quot-gr} how this material could be applied to study filtered deformations of graded quotient algebras in monoidal categories. 
Recall that we assume Hypotheses~\ref{hyp:abelian} and \ref{hyp:strict} throughout, and recall  the terminology below.


\begin{definition}[$f_1 \square f_2$] 
Let $f_1:X_1\to Y_1$ and $f_2: X_2 \to Y_2$ be morphisms in $\C$. We define their {\it pushout product} to be the unique morphism $f_1 {\tiny \square} f_2$ fitting into the commutative diagram below.
\[
\xymatrix@R=1.5pc{
X_1 \otimes X_2 \ar[r]^{f_1 \otimes \id_{X_2}} \ar[d]_{\id_{X_1} \otimes f_2} & Y_1 \otimes X_2 \ar[d] \ar@/^1.6pc/[rdd]^(.6){\id_{Y_1} \otimes f_2}& \\
X_1 \otimes Y_2 \ar[r] \ar@/_1.4pc/[rrd]_(.3){f_1 \otimes \id_{Y_2}}& \pushoutcorner (X_1 \otimes Y_2) +_{X_1 \otimes X_2} (Y_1 \otimes X_2) \ar@{..>}[rd]_{f_1 \square f_2}& \\
                &                 & Y_1 \otimes Y_2
}
\]
\end{definition}

\medskip

Since our monoidal category is assumed to be biexact, in particular, right exact in each slot, we have the following result.

\begin{lemma}\cite[Lemma~4.8]{RV} \label{lem:coker-pp}
We have that $\coker(f_1 \square f_2) \cong \coker(f_1) \otimes \coker(f_2)$, for any morphisms $f_1$ and $f_2$ in $\C$. \qed
\end{lemma}

Next, we introduce quotient algebras in monoidal categories via the result below; this result is \cite[Proposition~2.8]{BD} in the framework above.

\begin{proposition}[$A/I$, $\pi_I$, $\overline{m}$, $\overline{u}$]  \label{prop:quotalg}
Take $(A,m,u)$ in ${\sf Alg}(\C)$ with weak ideal $(I,\lambda_I, \rho_I,\phi_I)$ of $A$ in $\C$. 
Denote
$$A/I:=\coker(\phi_I), \quad \quad  \; \; \pi_I:A \to A/I \;\text{ \textnormal{(}canonical epi\textnormal{)}}.$$
Then, there exists a unique morphism $\overline{m}: A/I \otimes A/I \to A/I$ in $\C$, where 
\begin{equation} \label{eq:mult-quot}
\pi_I \; m = \overline{m}(\pi_I \otimes \pi_I),
\end{equation} 
along with $\overline{u} := \pi_I \; u: \unit \to A/I$, so that $(A/I, \; \overline{m}, \; \overline{u})$ is an algebra in $\C$.
\end{proposition}

\begin{proof}
Consider the following diagram.
{\small
\[
\xymatrix@R=1.2pc@C=4.5pc{
I \otimes I \ar[r]^{\phi_I \otimes \id} \ar[d]_{\id \otimes \phi_I} & A \otimes I \ar[d] \ar@/^1.6pc/[rdd]^(.6){\id \otimes \phi_I}& &\\
I \otimes A \ar[r] \ar@/_1.4pc/[rrd]_(.3){\phi_I \otimes \id}& \pushoutcorner (A \otimes I) +_{I \otimes I} (I \otimes A) \ar[rd]_{\phi_I \square \phi_I}& & \\
                &                 & A \otimes A \ar[r]^(.55){\coker(\phi_I \square \phi_I)} \ar[d]_m& C \ar@{..>}[d]^{\overline{m}}\\
                && A \ar[r]^{\pi_I} & A/I
}
\]
}

Here, $C \cong A/I \otimes A/I$ and $\coker(\phi_I \square \phi_I) = \pi_I \otimes \pi_I$ by Lemma~\ref{lem:coker-pp}. Moreover, $\overline{m}$ exists as pictured above by the universal property of cokernels. Indeed, 
$$\pi_I\; m(\id_A \otimes \phi_I) = \pi_I \; \phi_I\; \lambda_I \;=\; 0 \;  = \pi_I \; \phi_I\; \rho_I = \pi_I\; m(\phi_I \otimes \id_A),$$
and thus, the unique pushout morphism from $(A \otimes I) +_{I \otimes I} (I \otimes A)$ to $A/I$ is the zero map. So, $\pi_I \; m(\phi_I \square \phi_I) =0$, and $\overline{m}$ exists as claimed.

Now it suffices to show that $\overline{m}$ is associative and $\overline{u}$ is unital with respect to $\overline{m}$. Consider the following calculation:
\[
\begin{array}{rlll}
\smallskip
    \overline{m}(\overline{m} \otimes \id_{A/I})(\pi_I^{\otimes 3}) & = \overline{m}(\pi_I \otimes \pi_I)(\id_A \otimes m) & = \pi_I \;m\;(\id_A \otimes m) &\\
    \smallskip
  & = \pi_I \;m\;(m \otimes \id_A)  & = \overline{m}(\pi_I \otimes \pi_I)(m \otimes \id_A) \\
  & = \overline{m}(\id_{A/I} \otimes \overline{m})(\pi_I^{\otimes 3}).
\end{array}
\]
Here, the third equation uses the associativity of $m$ and the rest of the equations use \eqref{eq:mult-quot}.
Now  $\overline{m}$ is associative as $\pi_I^{\otimes 3}$ is epic. Moreover,
$$\overline{m}(\overline{u} \otimes \id_{A/I}) (\id_{\unit} \otimes \pi_I)= \overline{m}(\pi_I \otimes \pi_I)(u \otimes \id_A) = 
\pi_I \; m\;(u \otimes \id_A) = \pi_I,$$
where the last equation holds as $u$ is left unital with respect to $m$. Since $\pi_I$ is epic, the left unital condition holds for $A/I$. Likewise, the right unital condition holds for $A/I$.
\end{proof}

\begin{definition}
We refer to the algebra $(A/I,\; \overline{m}, \; \overline{u})$ in Proposition~\ref{prop:quotalg} as the {\it quotient algebra of $A$ by the weak ideal $I$}. 
\end{definition}




Now we show how quotient algebras are related via monoidal functors.

\begin{proposition} \label{prop:quot-functor}  Let $(F, F_2, F_0): \mathcal{C} \to \mathcal{D}$ be a right exact, monoidal functor. Take an algebra $A$ with weak ideal $(I, \phi_I)$ in $\C$. Then, $F(A)/F(I)$ and $ F(A/I)$ are isomorphic as algebras in $\D$.
\end{proposition}

\begin{proof}
Consider the commutative diagram below:
\begin{equation} \label{eq:muF}
\xymatrix@R=.7pc@C=8pc{
&& F(A)/F(I) \ar@/_.5pc/@{..>}[dd]_{\omega_F}\\
F(I) \ar[r]^{\phi_{F(I)} \;  \overset{\textnormal{Prop.\ref{prop:mon-preserve}(c)}}{=} \; F(\phi_I)} & F(A) \ar[ru]^{\pi_{F(I)}} \ar[rd]_{F(\pi_I)}&  \\
&& F(A/I) \ar@/_.5pc/@{..>}[uu]_{\omega'_F}
}
\end{equation}
Here, $\omega_F$ is the unique morphism that makes the diagram commute due to the universal property of the cokernel map $\pi_{F(I)}$. Moreover, right exact functors commute with cokernels, so $F(\pi_I) = F(\coker(\phi_I)) = \coker(F(\phi_I))$. Thus, $\omega'_F$ is the unique morphism that makes the diagram commute due to the universal property of the cokernel map $F(\pi_I)$. Now by the uniqueness of $\omega_F$ and $\omega'_F$, we must have that $\omega'_F \; \omega_F = \id_{F(A)/F(I)}$ and $\omega_F \; \omega'_F = \id_{F(A/I)}$. Thus, $\omega_F$ is an isomorphism in $\D$.

Next, we verify that $\omega_F$ is an algebra morphism. We compute:
\[
\begin{array}{rl}
\smallskip
\omega_F \; m_{F(A)/F(I)} \; (\pi_{F(I)} \otimes \pi_{F(I)}) 
& = \omega_F \; \pi_{F(I)} \; m_{F(A)}\\
\smallskip
&= F(\pi_I) \; m_{F(A)}\\
\smallskip
&= m_{F(A/I)} \; [F(\pi_I) \otimes F(\pi_I)]\\
\smallskip
&= m_{F(A/I)} \; [\omega_F \; \pi_{F(I)} \otimes \omega_F \; \pi_{F(I)} ]\\
\smallskip
& = m_{F(A/I)} \; (\omega_F \otimes \omega_F) \; (\pi_{F(I)} \otimes \pi_{F(I)})
\end{array}
\]
The first equation holds by \eqref{eq:mult-quot}; the second and fourth equations hold by \eqref{eq:muF}; the third equation holds since $F(\pi_I)$ is an algebra map (indeed, $\pi_I$ is an algebra map by Proposition~\ref{prop:quotalg} and $F$ is monoidal); and the last equation follows from a rearrangement of terms. Since $\pi_{F(I)} \otimes \pi_{F(I)}$ is epic, we obtain that $\omega_F$ is multiplicative. 
Moreover, we compute:
\[
\begin{array}{rlll}
\smallskip
\omega_F \;u_{F(A)/F(I)} &= \omega_F \;\pi_{F(I)} \;u_{F(A)} &= \omega_F \; \pi_{F(I)} \; F(u_A)  \;F_0 &= F(\pi_I)\;  F(u_A) \; F_0 \\
&= F(\pi_I \; u_A) \; F_0
&=  F(u_{A/I}) \;  F_0 &= u_{F(A/I)}.
\end{array}
\]
Here, the first and fifth equations hold by Proposition~\ref{prop:quotalg}; the second and last equations hold as $F$ is monoidal; and the third equation follows from \eqref{eq:muF}.  So, $\omega_F$ is unital, and thus, $\omega_F$ is an algebra morphism, as required.
\end{proof}

\begin{corollary} \label{cor:quot-gr}
If $A$ is a filtered algebra in $\C$, and $I$ is a filtered weak ideal of $A$ in~$\C$, then 
$$\gr(A)/\gr(I) \; \cong \; \gr(A/I)$$
as graded algebras in $\C$.
\end{corollary}

\begin{proof}
We have that $\gr$ is a right exact, monoidal functor from ${\sf Fil}(\C)$ to ${\sf Gr}(\C)$ by Theorem~\ref{thm:gr} and Proposition~\ref{prop:gr-leftadj}. So the result follows from Corollary~\ref{cor:gr-fil-epic}  and Proposition~\ref{prop:quot-functor}.
\end{proof}

Next, we consider filtered deformations of quotient algebras in $\C$. To do so, consider the construction below.

\begin{definition}[$(-)^f$] \label{not:fil-tensor}
Define the functor $$(-)^f: {\sf Gr}(\C) \longrightarrow {\sf Fil}(\C)$$ to be given by $(\coprod_{i \in \mathbb{N}_0} X_i)^f = (X, F_X)$, for $X = \coprod_{i \in \mathbb{N}_0} X_i$ and $F_X(j) = \coprod_{i =0}^j X_i$. 
\end{definition}

\begin{lemma}\label{lem:FilFunctor}
The canonical filtration functor $F:=(-)^f$ is monoidal with
$F_2$ given by inclusion morphisms and $F_0$ given by the identity morphism.
\qed
\end{lemma}

\begin{remark} \label{rem:quot-gr}
Note that  $\gr\; (-)^f$ is the identity functor on ${\sf Gr}(\C)$. Now if $B \in {\sf Alg}({\sf Gr}(\C))$, then $A := B^f \in {\sf Alg}({\sf Fil}(\C))$ by Lemma \ref{lem:FilFunctor}. In this case,   $A/I$ is a filtered deformation of $B/\gr(I)$ by Corollary~\ref{cor:quot-gr}. But, as in the case for $\C = \Vec$, computing $\gr(I)$ can be tedious; Poincar\'{e}-Birkhoff-Witt theorems and related homological methods are used to address this problem \cite{SW}. It would be interesting to develop such techniques to study filtered deformations of graded quotient algebras in monoidal categories.
\end{remark}



\section{Frobenius algebras in rigid monoidal categories} \label{sec:newDefn}

In this section, we  provide equivalent conditions for an algebra in a rigid monoidal category $\C$ [Definition~\ref{def:rigid}] to admit the structure of a Frobenius algebra in $\C$. This builds on work of Fuchs-Stigner  \cite{FS}. Recall that all monoidal categories in this work are assumed to be abelian, strict with $\otimes$ biexact [Hypotheses~\ref{hyp:abelian}, \ref{hyp:strict}]. Consider the terminology below.

\begin{definition}[$\ev_X$, $\coev_X$, $\ev'_X$, $\coev'_X$]  \cite[Section~2.10]{EGNO} \label{def:rigid}
An object $X$ in a monoidal category $\C$ is called \textit{rigid} if it has left and right duals. Namely, there exist objects $X^*$ and ${}^*X \in \C$ with  co/evaluation maps,
\begin{align*}
\ev_X: X^* \otimes X \rightarrow \mathbbm{1}, & \hspace{2cm} \coev_X: \mathbbm{1} \rightarrow X \otimes X^*,  \\
\ev'_X: X \otimes \leftdual X \rightarrow \mathbbm{1}, & \hspace{2cm} \coev'_X: \mathbbm{1} \rightarrow \leftdual X \otimes X, 
\end{align*} 
so that $(\id_{X} \otimes \ev_X)(\coev_X \otimes \id_{X})$, $(\ev_X \otimes \id_{X^*})(\id_{X^*} \otimes \coev_X)$, $(\ev'_X \otimes \id_{X})(\id_{X} \otimes \coev'_X)$, and $(\id_{\leftdual X} \otimes \ev'_X)(\coev'_X \otimes \id_{\leftdual X})$ are all identity morphisms. Moreover, $\C$ is called {\it rigid} if all of its objects are rigid.
\end{definition}

\begin{remark}
When $\C$ is an abelian, rigid monoidal category, we do not need the assumption that $\otimes$ is biexact as this is implied by \cite[Proposition~4.2.1]{EGNO}. 
\end{remark}

 Now we present  the main result of this section.

\begin{theorem}\label{thm:defFrob}
Take $\C$ a rigid monoidal category, and take $(A,m,u)\in {\sf Alg}(\C)$. Then the following conditions are equivalent:

\begin{enumerate}[(a),font=\upshape]
    \item There exist morphisms $\Delta:A \rightarrow A \otimes A$ and $\ep: A \rightarrow \mathbbm{1}$ in $\C$ such that $(A,m,u,\Delta,\ep)$ is in ${\sf FrobAlg(\C)}$.
    
    \smallskip
    
    \item There exist morphisms $p: A \otimes A \rightarrow \mathbbm{1}$ and $ q: \mathbbm{1} \rightarrow A \otimes A$ in $\C$ such that 
    \[ \hspace{1cm}  p (m \otimes \id_A) =  p (\id_A \otimes m), \hspace{0.5cm} (p \otimes \id_A)(\id_A \otimes q) = \id_A = (\id_A \otimes p)(q \otimes \id_A). \]

    \item There exists an isomorphism $\Phi_l: A \rightarrow \leftdual A$ of left $A$-modules in $\C$, with left $A$-action maps  $\lambda_A = m$ and $\lambda_{\leftdual A} = (\id_{\leftdual A} \otimes \ev'_A)(\id_{\leftdual A} \otimes m \otimes \id_{\leftdual A})(\coev'_A \otimes \id_{A \otimes \leftdual A})$.
    
    \medskip
    
    \item There exists an isomorphism $\Phi_r: A \rightarrow A^*$ of right $A$-modules in $\C$, with right $A$-action maps $\rho_A = m$ and $\rho_{A^*} = (\ev_A \otimes \id_{A^*})(\id_{A^*} \otimes m \otimes \id_{A^*})(\id_{A^* \otimes A}\otimes \coev_A)$.
    
    \medskip
    
    \item There exists a morphism $\nu: A \rightarrow \mathbbm{1}$ in $\C$ so that, if a left or right weak ideal $(I, \lambda_I, \phi_I)$ of $A$ factors through $\ker(\nu)$, then $\phi_I$ is a zero morphism in $\C$.

    \medskip
    
    \item There exists a morphism $\nu: A \rightarrow \mathbbm{1}$ in $\C$ so that, if a left or right  ideal $(I, \lambda_I, \phi_I)$ of $A$ factors through $\ker(\nu)$, then $\phi_I$ is a zero morphism in $\C$.
\end{enumerate}
\end{theorem}

\begin{definition} \label{def:Frobform}
 In the theorem above, we refer to the map $p$ in part~(b) as a {\it nondegenerate pairing}  of $A$ with {\it copairing} $q$. We also refer to the map  $\nu$ in part (e) (resp., part (f)) as a {\it weak Frobenius form} (resp., a {\it  Frobenius form}) on $A$.
\end{definition}

\begin{proof}[Proof of Theorem~\ref{thm:defFrob}]
The equivalence of (a) and (b) is well-known; see, e.g., \cite[Proposition~8]{FS}. The equivalence of (b) and (c) holds by \cite[Proposition~9]{FS}, and the equivalence of (c) and (d) follows from \cite[Lemma~5]{FS}.

\smallskip

Next, we show that (b) implies (e). With the pairing $p:A \otimes A \to \unit$ in part (b), define $\nu:=p(u \otimes \id_A): A \to \unit$. Now part (e) for left weak ideals $(I, \lambda_I, \phi_I)$ of $A$ holds by the following commutative diagram.
\[
\xymatrix@R=1.7pc@C=10pc{
A \otimes \ker(\nu) \ar@{}[ddr]^(.25){(2)} \ar[r]^{\id_A \otimes \iota} \ar@/^2pc/[rr]^{0} & A \otimes A \ar@{}[rd]|(.3){(4)}  \ar[r]^{\id_A \otimes \nu} & A \otimes \one\\
A \otimes I \ar@{}[r]|(.5){(3)} \ar[u]^{\id_A \otimes \overline{\phi}_I} \ar[ru]_{\id_A \otimes \phi_I}&  A \otimes A \otimes A \ar@{}[r]|(.4){(5)} \ar@{}[u]|(1.35){(1)} \ar[u]^{\id_A \otimes m} \ar[ru]_{\id_A \otimes p} &\\
A \otimes A \otimes I \ar@{}[r]|(.5){(6)} \ar[u]^{\id_A \otimes \lambda_I} \ar[ru]_{\id_A \otimes \id_A \otimes \phi_I}& A \ar[u]^{q \otimes \id_A} \ar@/_1pc/[ruu]_{\id_A}&\\
I \ar[u]^{ q \otimes \id_A} \ar[ru]_{\phi_I}& &
}
\]

Here, $\iota$ is the natural inclusion map, so (1) commutes by the definition of a kernel. By the hypothesis in part (e) on the left weak ideal  $(I, \phi_I)$, there exists a morphism $\overline{\phi}_I: I \to \ker(\nu)$ so that (2) commutes. The diagram (3) commutes as $I \in {}_A \C$, where $\lambda_A = m$. Diagram (4) commutes because $ p=  \nu  m: A \otimes A \to \one$ via the unit axiom. Now by part (b), there exists a morphism $q: \one \to A \otimes A$ in $\C$ so that (5) commutes. Moreover, the diagram (6) clearly commutes. Using this, we conclude that the the outer diagram commutes, and thus $\phi_I = 0$, as desired.

Likewise, (b) implies (e) for right  weak ideals by using a similar commutative diagram  with the hypothesis that $(p \otimes \id_A)(\id_A \otimes q) = \id_A$, for $p:= \nu  m$.

\smallskip

Next, (e) clearly implies (f).

\smallskip

Finally, we verify that  (f) implies (c). 
As in \cite{FS} take 
$$\Phi_l:=(\id_A \otimes  \nu  m)(\coev'_A \otimes \id_A): A \to \leftdual A.$$
In fact, $\Phi_l \in {}_A \C$, due to the following computation:
\[
\begin{array}{rl}
\smallskip
\lambda_{\leftdual A}(\id_A \otimes \Phi_l) \hspace{-.1in}
&= [\id_{\leftdual A} \otimes \nu m (\ev'_A \otimes \id_{A \otimes A})(\id_{A} \otimes \coev'_A \otimes \id_A)(m \otimes \id_A)](\coev'_A \otimes \id_{A \otimes A})\\
\smallskip
&= [\id_{\leftdual A} \otimes  \nu m (m \otimes \id_A)](\coev'_A \otimes \id_{A \otimes A})\\
\smallskip
&= [\id_{\leftdual A} \otimes  \nu m (\id_A \otimes m)](\coev'_A \otimes \id_{A \otimes A})\\
\smallskip
&= (\id_{\leftdual A} \otimes \nu m)(\coev'_A \otimes \id_{A}) m\\
&= \Phi_l \lambda_A.
\end{array}
\] 
\noindent Here, the first and fourth equation hold by commutativity of maps; the second equation holds by a rigidity axiom; and the third equation holds since $m$ is associative.
Set the notation
$$K:=\ker(\Phi_l) \quad \quad \text{and} \quad \quad C:=\coker(\Phi_l),$$
and it suffices to show that $K=0$ and $C = 0$. 

To get that $K=0$, we will  show that the mono $k: K \hookrightarrow A$ attached to $K$ is the zero morphism. To proceed, define $\lambda_K: A\otimes K \rightarrow K$ using the universal property of kernels as follows:
\[
\xymatrix@R=1.5pc@C=5pc{
A \otimes K \ar[r]^{\id_A \otimes k} \ar@{..>}[d]^{\lambda_K} & A \otimes A \ar[r]^{\id_A \otimes \Phi_l} \ar[d]^{\lambda_A=m} & A \otimes \leftdual A \ar[d]^{\lambda_{\leftdual A}}\\
 K \ar[r]^k & A \ar[r]^{\Phi_l} &  \leftdual A.
 }
\]
Here, the right square commutes due to $\Phi_l \in {}_A \C$, and the left square commutes due to the definition of a kernel.  Furthermore, 
\[
\begin{array}{rl}

\smallskip

k \; \lambda_K \; (m\otimes \text{id}_A) & = m \; (\text{id}_A \otimes k) \; (m \otimes \text{id}_A)  \\

\smallskip

& = m \; (m\otimes \text{id}_A) (\text{id}_A \otimes \text{id}_A \otimes k)\\

\smallskip

& = m \; ( \text{id}_A 
 \otimes m)(\text{id}_A \otimes \text{id}_A \otimes k) \\

\smallskip

& = m\; (\text{id}_A \otimes k) (\text{id}_A \otimes \lambda_K) \\

\smallskip

& = k \;  \lambda_K\; (\text{id}_A \otimes \lambda_K).

\end{array}
\]
\noindent
The first, fourth and fifth equations hold by definition of $\lambda_K$; the third equation holds by associativity of $m$, and the second equation holds by commutativity of maps.
Since $k$ is a mono, $\lambda_K (\text{id}_A \otimes \lambda_K) = \lambda_K (m\otimes \text{id}_A)$. Similarly, one can show that $\lambda_K (u\otimes \text{id}_K) = \text{id}_K $. Thus, $(K,k,\lambda_K)$ is a  left ideal of $A$. Moreover, we get 
\[
\begin{array}{rl}

\smallskip

0 &=\ev'_A(u \otimes \id_{\leftdual A})\; \Phi_l \; k\\

\smallskip

 &= \ev'_A(u \otimes \id_{\leftdual A})(\id_A \otimes \nu  m)(\coev'_A \otimes \id_A) k\\
 
\smallskip

  &= \nu m(\ev_{A}' \otimes \id_A \otimes \id_A)(\id_A \otimes \coev_{A}' \otimes \id_A)(u \otimes \id_A) k\\
  
\smallskip

  &= \nu m(u \otimes \id_A) k\\
  
  &= \nu  k.
 \end{array}
\]
Here, the first equation holds because $\Phi_l  k = 0$; the second equation follows from the definition of $\Phi_l$; the third equation holds by commutativity of maps; the fourth equation follows from rigidity; and the last equation holds by unitality.
So, $ \nu k: K \to \one$ is a zero morphism, which implies that $k$ factors through $\ker(\nu)$. Thus, by part (f), $k$ is the zero morphism. 

To obtain that $C=0$, consider the natural epi $c: \leftdual A \twoheadrightarrow C $, along with its monic dual, $c^*: C^* \hookrightarrow (\leftdual A)^*$. In particular, $(\leftdual A)^* = A$, and it is straightforward to show $C^* = \ker(\Phi_l^*)$. Then, by an argument similar to showing that $(K,k,\lambda_K)$ is a left  ideal of $A$ above, we obtain that $(C^*, c^*, \rho_{C^*})$ is a right  ideal of $A$. Here, the right $A$-module map $\rho_{C^*}$ is induced by the map $\rho_{A^*}$ given the statement of part (d). 
Moreover, by using the rigidity and the unit axioms, we also obtain from $\Phi_l^* c^* = 0$ that $\nu c^* = 0$. So, by part (f), we conclude that $c^* = 0$. Thus, $C^* = 0$, and hence, $C=0$, as desired.
\end{proof}


\section{Filtered Frobenius algebras in rigid monoidal categories} \label{sec:main}
We now present the main result of the paper on filtered Frobenius algebras in rigid monoidal categories [Theorem~\ref{thm:main}]; see Section~\ref{sec:thmmain}. First, we discuss  preliminary results in Section~\ref{sec:connected} on Frobenius forms of certain graded algebras that are Frobenius. We end by presenting questions for further investigation in Section~\ref{sec:questions}. Let $\C$ be an abelian, rigid monoidal category throughout, which is strict with $\otimes$ biexact by Hypothesis~\ref{hyp:abelian} and \ref{hyp:strict}.

\subsection{Frobenius graded algebras} \label{sec:connected} 

Consider the following terminology.

\begin{definition} \label{def:con}
\begin{enumerate}
    \item A graded algebra $B = \coprod_{i \in \mathbb{N}_0} B_i$  in $\C$ is {\it connected} if $B_0 = \unit$.
    \smallskip
    \item A filtered algebra $(B, F_B)$ in $\C$ is called {\it connected} if $F_B(0) = \unit$.
\end{enumerate}
\end{definition}

\begin{remark}
It is straight-forward to see that the associated graded algebra [Definition~\ref{def:assgrstr}] of a connected filtered algebra in $\C$ is a connected graded algebra in $\C$.
\end{remark}

Next, we have a preliminary result on the structure of connected graded algebras that are Frobenius.

\begin{lemma} \label{lem:Frobgr}
Take a connected graded algebra $B = \coprod_{i = 0}^n B_i$ in $\C$ that is Frobenius. Then the following statements hold.
\begin{enumerate} [(a),font=\upshape]
    \item $B_n = B_0^* = \unit$.
    \item $\ep: B \to \unit$ defined as the composition
    $$\ep: B \longrightarrow \textstyle B/ \coprod_{i = 0}^{n-1} B_i \overset{\sim}{ \longrightarrow} B_n = \unit$$ is a (weak) Frobenius form on $B$.
    \item $B_{n-i} = B_{i}^*$ for $0\leq i \leq n$. 
\end{enumerate}
\end{lemma}

\begin{proof}
(a) Since the algebra $(B,m,u)$ is Frobenius, we have a Frobenius form $\ep': B \to \unit$ so that $p = \ep'm: B \otimes B \to \unit$ is a nondegenerate pairing on $B$ with copairing $q: \unit \to B \otimes B$ [Theorem~\ref{thm:defFrob}, Definition~\ref{def:Frobform}]. Now let 
$$a_i: B_i \to B \qquad \text{and} \qquad f_i:B \to B_i$$
be the natural inclusion and projection maps from the decomposition $B = \coprod_{i = 0}^{n} B_i$. Namely, 
\begin{equation} \label{eq:coprod-decomp}
\textstyle \coprod_{i=0}^n a_i f_i = \id_B \qquad \text{and} \qquad f_i a_j = \delta_{i,j} \id_{B_i}.
\end{equation}
Next, consider the following computation:
$$
\textstyle \coprod_{i,j} (\ep'm \otimes \id_{B_n})(a_n \otimes a_i \otimes f_n a_j) (\id_{B_n} \otimes f_i \otimes f_j) (\id_{B_n} \otimes q) 
 \overset{\textnormal{\eqref{eq:coprod-decomp}}}{=} 
(p \otimes f_n) (a_n \otimes q) 
= f_na_n 
 \overset{\textnormal{\eqref{eq:coprod-decomp}}}{=} \id_{B_n}.
$$
Since $f_n a_j = \delta_{n,j} \id_{B_n}$, we must have that $j = n$ in the equation above. On the other hand, $m$ is a graded algebra map. Thus, $\text{im}(m(a_n \otimes a_i))$ is a subobject of $B_{n+i}$. Since, $B_k=0$ for $k>n$, we must have that $i =0$ in the equation above. Thus,
$$
[(p \otimes \id_{B_n})(a_n \otimes a_0 \otimes \id_{B_n})][ (\id_B \otimes f_0 \otimes f_n) (\id_{B_n} \otimes q)] 
= \id_{B_n}.
$$
Likewise, we also have that
$$
[(\id_{B_0} \otimes p)(\id_{B_0} \otimes a_n \otimes a_0)][ (f_0 \otimes f_n \otimes \id_{B_0}) (q \otimes \id_{B_0})] 
= \id_{B_0}.
$$
Therefore the maps $p(a_n \otimes a_0):B_n \otimes B_0 \to \unit$ and $(f_0 \otimes f_n)q: \unit \to B_0 \otimes B_n$ give $B_n$ the structure of the left dual $B_0^*$ of $B_0$. So, by the connected assumption on $B$, we then get that
$$B_n = B_0^* = \unit^* = \unit.$$

(b) Note that $\ker(\ep) = \textstyle \coprod_{i=0}^{n-1} B_i$. So to show that $\ep$ is a weak Frobenius form on $B$, we must verify that $\textstyle \coprod_{i=0}^{n-1} B_i$ does not have a nonzero left weak ideal of $B$. By way of contradiction, suppose that $(I, \phi_I)$ is a nonzero left weak ideal of $B$ so that there is a map $\overline{\phi}_I: I \to \ker(\ep)$ with $\phi_I = \iota \;\overline{\phi}_I$; here, $\iota$ the natural mono from $\ker(\ep)$ to $B$. Then, by the definition of left weak ideals, for the left $B$-action maps $\lambda_I: B \otimes I \to I$ and $\lambda_B = m: B \otimes B \to B$, we get that $m(\id_B \otimes \phi_I) = \phi_I \lambda_I$. So, $m(\id_B \otimes \phi_I) = \iota\; \overline{\phi}_I\; \lambda_I$. On one hand,  $m$ is a graded map, so we must have that the image of $m(\id_B \otimes \phi_I)$ has $B_n$ as a component. On the other hand, the image of $\iota\; \overline{\phi}_I\; \lambda_I$ does not have $B_n$ as a component, which yields a contradiction. Hence, $\ker(\ep)$ does not have a nonzero left weak ideal of $B$. Thus, with part (a) we obtain that $\ep: B \to \textstyle  B/\coprod_{i=0}^{n-1} B_i \cong B_n = \unit$ is a weak Frobenius form on $B$.

\smallskip

(c) By part (b), $B$ is Frobenius with weak Frobenius form $\ep$. Hence $p' = \ep m: B \otimes B \rightarrow \unit$ is a nondegenerate pairing on $B$ for some copairing $q':\unit \rightarrow B\otimes B$. Now using the same argument as in part (a), we can show that 
\[ p'(a_{n-i} \otimes  a_i): B_{n-i}\otimes B_i \rightarrow \unit \hspace{.5cm} \text{and} \hspace{.5cm} (f_i \otimes f_{n-i})q': \unit \rightarrow B_{i} \otimes B_{n-i},  \] 
give $B_{n-i}$ the structure of the left dual $B_{i}^*$ of $B_{i}$.
\end{proof}

\subsection{Main result} \label{sec:thmmain}
This brings us to the main result of the article.

\begin{theorem} \label{thm:main} Take $A$ to be a connected filtered algebra in $\C$ equipped with a finite monic filtration.
If $\gr(A)$ is a Frobenius algebra in $\C$, then so is~$A$. 
\end{theorem}

\begin{proof}
Since the filtration on $A$ is finite, $A \cong F_A(n)$ for some $n \in \mathbb{N}_0$.
Recall Notation~\ref{not:filt} and consider the composite  morphism $$\eta: A \overset{\sim}{\longrightarrow} F_A(n) \overset{\pi_n^A}{\longrightarrow} \overline{F_A(n)}.$$ 
Since $\gr(A)$ is a graded algebra in $\C$ by Theorem~\ref{thm:gr}, and is Frobenius by assumption, we get by Lemma \ref{lem:Frobgr}(a) that $\overline{F_A(n)} = \unit$. 
Let $(I,\lambda,\phi)$ be a left ideal of $A$, so that $\phi$ factors through $\ker(\eta) =F_A(n-1)$. Then, by Theorem~\ref{thm:defFrob}, it suffices to show that $\phi$ is a zero morphism. Indeed, this would show that $\eta$ is a Frobenius form for $A$.

Since $I$ factors through ker$(\eta)$, we have a map $\overline{\phi} : I \rightarrow F_A(n-1) $ such that $\iota^A_{n-1} \; \overline{\phi} = \phi$. 
By Proposition~\ref{prop:idealFil}, we can endow $I$ with a filtration $F_I$ making it a filtered left ideal of $A$. Recall from Lemma~\ref{lem:induced} that the filtration $F_I$ is defined using pullbacks. Consider the following diagram: 
\[
\xymatrix{
F_I(n) \ar@/^1pc/[rrd]^{\text{id}} \ar@/_2pc/[rdd]_{\overline{\phi}} \ar@{.>}[rd]^{\theta} \ar@{}[rrd]|(0.5){(2)} \ar@{}[rdd]|{(3)} & &\\
& F_I(n-1) \ar@{}[rd]|(0.5){(1)} \ar[d]_{F_{\phi}(n-1)}  \ar[r]^{\iota^I_{n-1}} & F_I(n) \ar[d]_{\phi} \\
& F_A(n-1) \ar[r]_{\quad \iota^A_{n-1}} & F_A(n).  
}
\]
Here, square $(1)$ commutes by definition of $F_I(n-1)$  (see the diagram in the proof of  Lemma~\ref{lem:induced}). Since $\iota^A_{n-1} \; \overline{\phi} = \phi \; \text{id}$, by universal property of pullbacks, we get a morphism $\theta: F_I(n) \rightarrow F_I(n-1)$ such that $(2)$ and $(3)$ commute. Thus, we get that $\iota^I_{n-1} \; \theta = \text{id}$. Hence, $\iota^I_{n-1}$ must be an epimorphism. Therefore, by \eqref{eq:gr-comp}, we have that $\gr(I)_n=0$. Thus,
the morphism $\gr(\phi): \gr(I)\rightarrow \gr(A) $ factors through $\oplus_{i=0}^{n-1} \overline{F_A(i)}$.

By Lemma \ref{lem:Frobgr}(b), we know that $\gr(A)$ is Frobenius with weak Frobenius form $\varepsilon$ and $\text{ker}(\varepsilon)=\oplus_{i=0}^{n-1} \overline{F_A(i)}$. 
Furthermore, by Proposition~\ref{prop:mon-preserve}(c), we know that $\gr(I)$ is a left weak ideal of $\gr(A)$.
Since the map gr($\phi$) from $\gr(I)$ to $\gr(A)$ factors through $\text{ker}(\varepsilon)$, by Theorem~\ref{thm:defFrob} we get that $\gr(\phi)=0$.

We claim that $\gr(\phi)_i$  is monic; this is verified in \cite[Lemma~4.3.2]{PP}, but we include the details for the reader's convenience. To proceed, consider the commutative diagram \eqref{eq:gr-comp} corresponding to the morphism $\phi$ as pictured below; recall we assume that $\iota_{j}^A$ is monic for all $j \in \mathbb{N}_0$. Moreover, consider the kernel $(K,k)$ of $\gr(\phi)_i$, and let $P$ be the pullback of $k$ and $\pi_i^I$, given by $\alpha: P \to K$ and $\beta: P \to F_I(i)$. Note that $\pi_i^A \; F_\phi(i) \; \beta = \gr(\phi)_i \;  \pi_i^I \; \beta = 0$. Since $\ker(\pi_i^A) = \textnormal{im}(\iota_{i-1}^A)$ and $\iota_{i-1}^A$ is monic, there exists a unique map $\gamma: P \to F_A(i-1)$ so that $F_\phi(i) \; \beta = \iota_{i-1}^A \; \gamma$. Since $F_I(i-1)$ is a pullback, there also exists a unique map $\delta: P \to F_I(i-1)$ so that $\iota_{i-1}^I \; \delta = \beta$. 
\[
\xymatrix{
&& P \pullbackcorner \ar[r]^\alpha \ar[d]_{\beta} \ar[dl]_{\delta} \ar@/_4.2pc/[ddl]_{\gamma}&K \ar[d]^{k}\\
&\hspace{.3in}F_I(i-1) \pullbackcorner \ar[r]^(.65){\iota^I_{i-1}} \ar[d]^{F_\phi(i-1)} & F_I(i) \ar[r]^{\pi^I_{i}} \ar[d]^{F_\phi(i)}          & \overline{F_I(i)} \ar[d]^{\gr(\phi)_i = 0} \ar[r] & 0 \\
0 \ar[r] &F_A(i-1) \ar[r]^(.55){\iota^A_{i-1}}                       & F_A(i) \ar[r]^{\pi^A_{i}}    & \overline{F_A(i)} \ar[r] & 0.             
}
\]
Now 
$k \alpha = \pi_i^I \; \beta = \pi_i^I \;\iota_{i-1}^I \;\delta = 0.$
Since $P$ is a pullback and $\pi_i^I$ is epic, the morphism $\alpha$ is epic as well \cite[Corollary~4.2.6]{PP}. As a result, $k = 0$, as required.

Now we show that $F_I(i) = 0$ for all $i$ via induction. Since $\phi = F_\phi(n)$ is monic, and $F_I(n-1)$ is constructed via the pullback of $\phi$ and $\iota_{n-1}^A$, we obtain that $F_\phi(n-1)$ is also monic. 
Likewise, $F_\phi(i)$ is monic for all $i$.  Thus, the map $F_\phi(0): F_I(0) \to F_A(0)$ is monic, and is zero as $F_\phi(0) = \text{gr}(\phi)_0=0$. Hence, $F_I(0)=0$. We now assume by induction that $F_I(i-1)=0$. As shown above,  $\gr(\phi)_i$ is a zero monic, so  $\overline{F_I(i)} = 0$. So we can conclude that  $F_I(i) = 0$. 

Therefore, $I=\text{colim}_i F_I(i) = 0$, and thus, $\phi = 0$, as required for $A$ to be Frobenius.
\end{proof}

\subsection{Further directions} 
\label{sec:questions}  We end this section by listing  directions for further investigation.

\smallskip

First, as mentioned in the introduction, Theorem~\ref{thm:main} is a categorical generalization of the main result of Bongale's 1967 work \cite{Bongale1}. In her 1968 work \cite{Bongale2}, Bongale generalized the 1967 result by removing the connected assumption. 

\begin{question} \label{ques:nonconn}
Does Theorem~\ref{thm:main} hold when $A$ is not necessarily connected?
\end{question}

Next, pertaining to the monoidal associated graded functor $\gr:{\sf Fil}(\C) \to {\sf Gr}(\C)$ from Theorem~\ref{thm:gr}, we inquire:

\begin{question}
Does there exist an adjoint to $\gr$ that admits the structure of a Frobenius monoidal functor (see, e.g., \cite{DayPas}), that can be used to obtain Theorem~\ref{thm:main} or more generally, to address Question~\ref{ques:nonconn} in the case when $\C$ is a rigid monoidal category? 
\end{question}

\noindent For instance, an adjoint to an associated graded functor is discussed in \cite{GKR}; their functor is slightly different than our functor $\gr$ in Section~\ref{sec:assgr}.
\smallskip


    

Moreover, in connection with the $\C = \Vec$ case discussed in the introduction, consider the following direction.

\begin{remark}
One can also (aim to define and) analyze generalizations of other conditions for algebras in (certain) monoidal categories $\C$, such as the integral domain, prime, Noetherian, and Calabi-Yau conditions, and study when these properties lift to a filtered algebra $A$ in $\C$ from the associated graded algebra $\gr(A)$ in $\C$.
\end{remark}

On the other hand,  Launois and Topley recently obtained a generalization of Bongale's results for {\it Frobenius extensions}, which are $\kk$-algebra extensions $S \subset R$ so that $R$ is a projective left $S$-module and $R \cong {\sf Hom}_S(R,S)$ as $(R,S)$-bimodules \cite{LT}. This recovers the classical definition of a Frobenius algebra when $S = \kk$. So we ask:

\begin{question}
Is there a generalization of \cite[Main Theorem]{LT} for the setting of filtered/graded algebras in monoidal categories as in this work?
\end{question}

Finally, one of the main uses of  Frobenius algebras in monoidal categories $\C$ in the ungraded setting is their role in classifying $\C$-valued 2-dimensional Topological Quantum Field Theories (2d-TQFTs); see, e.g., \cite{kock2004frobenius}. The latter is used to produce invariants in $\C$ for  2-dimensional oriented cobordisms between closed 1-dimensional manifolds. See the following remark about similar connections in graded and filtered  settings.

\begin{remark}  \label{rem:TFT} 
Graded and filtered Frobenius algebras in monoidal categories are used to examine various quantum field theories, as discussed below.
\begin{enumerate}
    \item By work of Turaev \cite[Section~III.3]{TuraevHQFT}, a variant of 2d-TQFTs, namely 2-dimensional Homotopy Quantum Field Theories (2d-HQFTs), are classified in terms of {\it crossed Frobenius algebras} that are graded by a group $G$.  In particular, the cobordisms are equipped with homotopy classes of maps into a certain space $X$, and the grading group $G$ is the fundamental group of $X$. So 2d-HQFTs, and  the $G$-graded Frobenius algebras above, produce invariants of 2-dimensional cobordisms with certain homotopical structure between closed 1-dimensional manifolds.
    \medskip
    \item By work of Lazaroiu \cite{Laz} and of Lauda-Pfeiffer \cite{LaudaPfeiffer}, another important variant of 2d-TQFTs, namely open-closed 2d-TQFTs, are examined by Frobenius algebras in the graded and filtered setting. In this case, the cobordisms are smooth compact oriented 2-dimensional manifolds with corners that model the  topology of {\it open and closed string worldsheets}. Gradings and filtrations are defined on these 2d-TQFTs in order to understand the homology of oriented tangles; see \cite[Proposition~2.16 and Theorem~3.37]{LaudaPfeiffer}. Thus, the grading and filtration on the corresponding Frobenius algebras are important for examining such topological invariants. 
\end{enumerate}

\end{remark}


\section{Application: On module categories over symmetric finite  tensor categories}

In this section, we present an application of our main result, Theorem \ref{thm:main}. Namely, we prove that every exact module category over a symmetric finite tensor category $\C$ is isomorphic to the category of modules over a Frobenius algebra $A$ in $\C$ [Theorem \ref{thm:repByFrob}].
Let $\kk$ be an algebraically closed field of characteristic zero. We refer the reader to \cite[Sections~1.8,~4.1,~7.1,~7.5,~7.8,~8.1]{EGNO} for details about the next categorical structures.

\begin{definition}  \label{def:sym-tensor}Let $(\C,\otimes,\unit, c)$ be a $\kk$-linear, abelian, rigid, braided monoidal category. We say that $\C$ is:
\begin{enumerate}
    \item {\it symmetric} if the braiding satisfies $c_{Y,X} \; c_{X,Y} = \id_{X \otimes Y}$ for all $X,Y \in \C$;
    \smallskip
    \item {\it finite} if it has finite dimensional spaces of morphisms and has finite length objects (i.e., is locally finite), has enough projectives, and has finitely many isomorphism classes of simple objects;
    \smallskip
    \item a {\it finite tensor category} if, further, $\otimes$ is bilinear on morphisms and ${\sf Hom}_\C(\unit,\unit) \cong \kk$;
     \smallskip
    \item a {\it fusion category} if $\C$ is a finite tensor category that is also semisimple. 
\end{enumerate}
\end{definition}

\begin{hypothesis}
From now on, take $\C$ to be a symmetric finite tensor category. 
\end{hypothesis}

\begin{definition} \label{def:mod-category}
A {\it left module category} over $\C$ is a locally finite, abelian category $\mathcal{M}$ equipped with the following:
\begin{itemize}
    \item a bifunctor  $\triangleright:  \C \times \mathcal{M} \to \mathcal{M}$ which is bilinear on morphisms and exact in first slot,
    \item  module associativity constraints  that satisfy the pentagon and triangle axioms. 
\end{itemize} 

Further, $\mathcal{M}$ is said to be {\it exact}, if, for any projective object $P \in \C$ and any object $M \in \mathcal{M}$, the object $P \triangleright M$ is projective in $\mathcal{M}$.
\end{definition}

\begin{example} \label{ex:modcat}
Take an algebra $A$ in $\C$, then one can form a category, $\C_A$, of right $A$-modules in $\C$ consisting of objects $M \in \C$ equipped with a right action morphism $\rho: M \otimes A \to M$ in $\C$. Moreover, $\C_A$ is a left module category over $\C$ via bifunctor $\C \times \C_A \to \C_A$ given by $(X,(M,\rho)) \mapsto (X \otimes M,\; \id_X \otimes \rho)$.  
\end{example}

In fact, the example above classifies all exact module categories over symmetric finite tensor categories, up to equivalence.

\begin{definitionproposition}\cite[Theorem~3.1]{ostrik2003module} \cite[Theorem~3.17]{etingof2003finite} \label{defprop:repAlg}
Every exact module category $\M$ over $\C$ is equivalent to a module category $\C_{A}$ in Example~\ref{ex:modcat}, for some $A \in {\sf Alg}(\C)$. In this case, we say that $\M$ is represented by $A$. \qed
\end{definitionproposition}

Building on this result, we establish the following statement.

\begin{theorem}\label{thm:repByFrob}
Every exact module category over a symmetric finite tensor category $\C$ is represented by a Frobenius algebra in $\C$.
\end{theorem}

To verify the theorem above, we now recall and establish some preliminary results.

\begin{lemma} \cite[Section~2.4.8]{kock2004frobenius} \label{lem:tens-Frob}
If $C,C' \in {\sf FrobAlg}(\C)$, then $C \otimes C' \in {\sf FrobAlg}(\C)$. 
\qed
\end{lemma}

Next, we recall Deligne's classification of symmetric finite tensor categories, and its Hopf-algebraic interpretation  by Andruskiewitsch-Etingof-Gelaki.

\begin{proposition}[${\sf Rep}(G \ltimes W,u)$, $R_u$]  \label{prop:Deligne}
\cite[Corollaries~0.7,~0.8]{deligne2002categories} \cite{AEG} Recall that $\C$ is a symmetric finite tensor category. 
\begin{enumerate} [font=\upshape]
    \item  Then, $\C$ is equivalent to a category of super-representations of a finite supergroup. \smallskip
    \item  Equivalently, $\C$ is a category of representations, ${\sf Rep}(G \ltimes W,u)$, consisting of representations of a triangular Hopf algebra $\Lambda(W) \# \kk G$, where $G$ is a finite group,  $u \in Z(G)$ with $u^2 = 1$ and $W$ is a $G$-representation satisfying $u\cdot w= -w,$ for all $ w\in W$. The braiding is given by R-matrix $R_u = \frac{1}{2}(1\otimes 1 + u \otimes 1 + 1 \otimes u - u \otimes u)$. \smallskip
    \item Further, ${\sf Rep}(G \ltimes W,u)$ is fusion precisely when $W = 0$. \qed
\end{enumerate}
\end{proposition}

Here, we freely identify representations with left modules. Next, consider the following preliminary results.

\begin{lemma}\label{lem:induction}
Let $H \leq \hat{H} \leq G$ be a sequence of groups, let $u\in \hat{H}$ be an element of order~$\leq 2$, and let $W$ be a representation of $G$ acting by $-1$ on $W$. Then the following induction functors  send Frobenius algebras to Frobenius algebras:
\begin{enumerate}[font=\upshape]
    \item $\textnormal{Ind}_H^{\hat{H}}: \Rep(H)\rightarrow \Rep(\hat{H})$, defined on objects by
    $$(U,\; \alpha: \kk H \otimes U \to U) \mapsto (\kk \hat{H} \otimes_H U, \;m_{\kk \hat{H}} \otimes_H \id_U);$$
    \item $(\textnormal{Ind}_{\hat{H}})_W: \Rep(\hat{H}) \rightarrow \Rep (\hat{H} \ltimes W ,u)$, defined on objects by
    $$(U,\; \alpha: \kk \hat{H} \otimes U \to U) \mapsto (U, \; \alpha, \; \beta:\Lambda(W) \otimes U \to U ~\text{trivial action});$$ 
    \item $(\textnormal{Ind}_{\hat{H}}^G )_W:\Rep(\hat{H}\ltimes W,u) \rightarrow \Rep(G\ltimes W,u)$ defined on objects by
    $$(U,\; \alpha: \kk \hat{H} \otimes U \to U, \; \beta:\Lambda(W) \otimes U \to U) \mapsto (\kk G \otimes_{\hat{H}} U, \; m_{\kk G} \otimes_{\hat{H}} \id_U, \; \id_{\kk G} \otimes_{\hat{H}} \beta).$$ 
\end{enumerate}
\end{lemma}

\begin{proof}
Each of the parts holds because the induction functors above are {\it Frobenius monoidal} \cite[Definition~1]{DayPas}; see, e.g., the proof of \cite[Proposition~B.1]{FHL}. Thus, these functors send Frobenius algebras to Frobenius algebras \cite[Corollary~5]{DayPas}.
\end{proof}

\begin{lemma}\label{lem:endVfrob}
Let  $V$ be finite dimensional representation of a twisted group algebra $\kk H_{\psi}$. Then, $\End(V)\in {\sf FrobAlg}( \Rep(H))$.
\end{lemma}

\begin{proof}
It is well known that $\End(V)\in \Rep(H)$ via the $H$-action $h\cdot f:= \sigma(h) \circ f \circ \sigma(h)^{-1}$. The algebra structure comes from multiplication given by composition and unit as $\id_V$. 
Suppose that $\dim_\kk(V) = n$. Then, after identifying $\End(V)$ with $\textnormal{Mat}_n(\kk)$, its basis is given by the elementary matrices $\{E_{i,j}\}_{1\leq i,j\leq n}$. By taking $\Delta(E_{i,j})=\sum_{k=1}^n E_{i,k} \otimes E_{k,j}$ and $\varepsilon(E_{i,j})=\delta_{i,j}$ and extending linearly to $\End(V)$, it is straight-forward to check that $(\End(V),m,u,\Delta,\varepsilon)\in {\sf FrobAlg}( \Rep(H))$. 
\end{proof}

These two lemmas yield a short proof of Theorem \ref{thm:repByFrob} in the fusion case, as we see next.

\begin{proposition}
Every exact module category over a symmetric fusion category is represented by a Frobenius algebra.
\end{proposition}

\begin{proof}
By Proposition~\ref{prop:Deligne}(c), any symmetric fusion category is equivalent as a braided fusion category to the category $\Rep(G,u)$ where $G$ is a finite group and $u\in G$ is a central element of order $\leq 2$. By Proposition~\ref{prop:Deligne}(b), $\Rep(G,u)$ is the category of finite dimensional representations of the triangular Hopf algebra $(\kk G,R_u)$ with $R$-matrix $R_u$. 
Now by \cite[Theorem~3.2]{ostrik2003module}, every indecomposable, exact module category over $\Rep(G,u)$ is  equivalent to $\Rep(\kk H_{\psi})$ for some $H \leq G$ and $\psi \in H^2(H,\kk^*)$. By \cite[Lemma~4.3]{etingof2003finite}, each such module category is represented by an algebra $\textnormal{Ind}_H^G (\End (V))$, where $V$ is an irreducible representation of $\kk H_{\psi}$. Therefore, the result holds by Lemmas \ref{lem:induction} and \ref{lem:endVfrob}. 
\end{proof}

\begin{lemma} \label{lem:exterior}
If  $W$ is a finite-dimensional representation of $G$, then the exterior algebra  $\Lambda(W)$ is a Frobenius algebra in ${\sf Rep}(G)$, and is also a Frobenius algebra in ${\sf Rep}(G\ltimes W,u)$.
\end{lemma}

\begin{proof}
The first statement is well-known; the second statement holds by Lemma~\ref{lem:induction}(b).
\end{proof}

This brings us to the proof of the main result of this section, Theorem \ref{thm:repByFrob}, above.

\medskip

\noindent
\textit{Proof of Theorem \ref{thm:repByFrob}}. 
By Proposition~\ref{prop:Deligne}(a,b), it suffices to take $\C = {\sf Rep}(G \ltimes W,u)$. Now by \cite[Theorem~4.5]{etingof2003finite}, any exact module category is represented by an algebra of the form:
$$A:=(\textnormal{Ind}^G_{\hat{H}})_W \left( (\textnormal{Ind}_{\hat{H}})_W( \textnormal{Ind}^{\hat{H}}_H(\textnormal{End}(V))) \otimes \textnormal{Cl}_W \right) $$
in $\C$, for some subgroup $H$ of $G$, for $\hat{H}$ being the subgroup of $G$ generated by $H$ and $u$, and for $V$ being some irreducible representation of a twisted group algebra $\kk H_\psi$. 
Moreover,  $\textnormal{Cl}_W $ is a Clifford algebra in ${\sf Rep}(\hat{H} \ltimes W,u)$, which by step (g) in the proof of \cite[Theorem~4.5]{etingof2003finite}, is a filtered deformation of an exterior algebra $\Lambda(W)$ in ${\sf Rep}(\hat{H} \ltimes W,u)$. That is, the associated graded algebra of $\textnormal{Cl}_W$ is equal to  $\Lambda(W)$ in ${\sf Rep}(\hat{H} \ltimes W,u)$. Here, Lemma~\ref{lem:exterior} applies to conclude that $\Lambda(W)$ is a Frobenius algebra in ${\sf Rep}(\hat{H} \ltimes W,u)$. Our main result of this work, Theorem~\ref{thm:main}, then implies that $\textnormal{Cl}_W \in {\sf FrobAlg}({\sf Rep}(\hat{H} \ltimes W,u))$.
On the other hand, $(\textnormal{Ind}_{\hat{H}})_W (\textnormal{Ind}^{\hat{H}}_H( \textnormal{End}(V))) \in {\sf FrobAlg}({\sf Rep}(\hat{H} \ltimes W,u))$ by applying Lemmas~\ref{lem:endVfrob} and~\ref{lem:induction}(a,b). So with Lemma~\ref{lem:tens-Frob}, we get that 
$(\textnormal{Ind}_{\hat{H}})_W(\textnormal{Ind}^{\hat{H}}_H(\textnormal{End}(V))) \otimes  \textnormal{Cl}_W$
is a Frobenius algebra in ${\sf Rep}(\hat{H} \ltimes W,u)$. The result now follows by applying Lemma~\ref{lem:induction}(c) to obtain that $A \in {\sf FrobAlg}({\sf Rep}(G \ltimes W,u))$.
\qed


\bibliography{biblio}
\bibliographystyle{alpha}

\end{document}